\documentclass[11pt,letter]{article}

\usepackage[utf8]{inputenc} 
\usepackage{url}            
\usepackage{wrapfig}
\usepackage[authoryear]{natbib}

\usepackage{url}
\usepackage[colorlinks=True,linkcolor=magenta,citecolor=blue,urlcolor=blue,pagebackref=true,backref=true]
{hyperref}
\renewcommand*{\backref}[1]{\ifx#1\relax \else Page #1 \fi}
\renewcommand*{\backrefalt}[4]{%
    \ifcase #1 \footnotesize{(Not cited.)}%
    \or        \footnotesize{(Cited on page~#2.)}%
    \else      \footnotesize{(Cited on pages~#2.)}%
    \fi}
\usepackage{graphicx}

\usepackage{amsmath,amssymb,fullpage,graphicx}
\usepackage{comment}
\usepackage{color}
\usepackage{setspace}
\usepackage{pdflscape}
\usepackage{enumitem}
\usepackage{wrapfig}

\usepackage[linesnumbered,ruled]{algorithm2e} 
\DontPrintSemicolon	

\SetKwFor{kwEstep}{E-Step}{:}{}%
\SetKwFor{kwMstep}{M-Step}{:}{}%
\SetKwInput{init}{Input}
\SetKwInput{output}{Output}

\usepackage{amssymb,amsmath,amsthm}
\usepackage[english]{babel}

\usepackage{macros}

\makeatletter
\renewcommand\tableofcontents{%
    \@starttoc{toc}%
}
\makeatother

\long\def\comment#1{}
\definecolor{battleshipgrey}{rgb}{0.52, 0.52, 0.51}
\definecolor{darkgray}{rgb}{0.66, 0.66, 0.66}
\definecolor{darkgreen}{rgb}{0.0, 0.2, 0.13}
\definecolor{darkspringgreen}{rgb}{0.09, 0.45, 0.27}
\definecolor{dukeblue}{rgb}{0.0, 0.0, 0.61}
\definecolor{olivedrab7}{rgb}{0.24, 0.2, 0.12}
\definecolor{darkblue}{rgb}{0.0, 0.0, 0.55}
\definecolor{darkscarlet}{rgb}{0.34, 0.01, 0.1}
\definecolor{candyapplered}{rgb}{1.0, 0.03, 0.0}
\definecolor{ao(english)}{rgb}{0.0, 0.5, 0.0}
\definecolor{applegreen}{rgb}{0.55, 0.71, 0.0}

\makeatletter
\let\save@mathaccent\mathaccent
\newcommand*\if@single[3]{%
  \setbox0\hbox{${\mathaccent"0362{#1}}^H$}%
  \setbox2\hbox{${\mathaccent"0362{\kern0pt#1}}^H$}%
  \ifdim\ht0=\ht2 #3\else #2\fi
  }
\newcommand*\rel@kern[1]{\kern#1\dimexpr\macc@kerna}
\newcommand*\widebar[1]{\@ifnextchar^{{\wide@bar{#1}{0}}}{\wide@bar{#1}{1}}}
\newcommand*\wide@bar[2]{\if@single{#1}{\wide@bar@{#1}{#2}{1}}{\wide@bar@{#1}{#2}{2}}}
\newcommand*\wide@bar@[3]{%
  \begingroup
  \def\mathaccent##1##2{%
    \let\mathaccent\save@mathaccent
    \if#32 \let\macc@nucleus\first@char \fi
    \setbox\z@\hbox{$\macc@style{\macc@nucleus}_{}$}%
    \setbox\tw@\hbox{$\macc@style{\macc@nucleus}{}_{}$}%
    \dimen@\wd\tw@
    \advance\dimen@-\wd\z@
    \divide\dimen@ 3
    \@tempdima\wd\tw@
    \advance\@tempdima-\scriptspace
    \divide\@tempdima 10
    \advance\dimen@-\@tempdima
    \ifdim\dimen@>\z@ \dimen@0pt\fi
    \rel@kern{0.6}\kern-\dimen@
    \if#31
      \overline{\rel@kern{-0.6}\kern\dimen@\macc@nucleus\rel@kern{0.4}\kern\dimen@}%
      \advance\dimen@0.4\dimexpr\macc@kerna
      \let\final@kern#2%
      \ifdim\dimen@<\z@ \let\final@kern1\fi
      \if\final@kern1 \kern-\dimen@\fi
    \else
      \overline{\rel@kern{-0.6}\kern\dimen@#1}%
    \fi
  }%
  \macc@depth\@ne
  \let\math@bgroup\@empty \let\math@egroup\macc@set@skewchar
  \mathsurround\z@ \frozen@everymath{\mathgroup\macc@group\relax}%
  \macc@set@skewchar\relax
  \let\mathaccentV\macc@nested@a
  \if#31
    \macc@nested@a\relax111{#1}%
  \else
    \def\gobble@till@marker##1\endmarker{}%
    \futurelet\first@char\gobble@till@marker#1\endmarker
    \ifcat\noexpand\first@char A\else
      \def\first@char{}%
    \fi
    \macc@nested@a\relax111{\first@char}%
  \fi
  \endgroup
}
\makeatother

\begin{document}

\begin{center} {\LARGE{\bf{Refined Convergence Rates for Maximum Likelihood \\[0.08in] Estimation under Finite Mixture Models}}}

\vspace*{.2in}
 {\large{
 \begin{tabular}{cc}
 Tudor Manole$^{\diamond}$ &  Nhat Ho$^{\ddagger}$ \\
 \end{tabular}
}}

\vspace*{.1in}

\begin{tabular}{c}
Department of Statistics and Data Science, Carnegie Mellon University$^{\diamond}$ \\ Department of Statistics and Data Sciences, University of Texas, Austin$^{\ddagger}$
\end{tabular}

\today

\vspace*{.2in}

\end{center}

\begin{abstract}
     We revisit the classical problem of deriving convergence rates for the maximum likelihood estimator (MLE) in finite mixture models. The Wasserstein distance has become a standard loss function for the analysis of parameter estimation in these models, due in part to its ability to circumvent label switching and to accurately characterize the behaviour of fitted mixture components with vanishing weights. However, the Wasserstein distance is only able to capture the worst-case convergence rate among the remaining fitted mixture components. We demonstrate that when the log-likelihood function is penalized to discourage vanishing mixing weights, stronger loss functions can be derived to resolve this shortcoming of the Wasserstein distance. These new loss functions accurately capture the heterogeneity in convergence rates of fitted mixture components, and we use them to sharpen existing pointwise and uniform convergence rates in various classes of mixture models. In particular,  these results imply that a subset of the components of the penalized MLE typically converge significantly faster than could have been anticipated from past work. We further show that some of these conclusions extend to the traditional MLE. Our theoretical findings are supported by a simulation study to illustrate these improved convergence rates. 
\end{abstract}

\section{Introduction}
Finite mixture models form a celebrated tool for modelling heterogeneous data, and are used pervasively in the life and 
physical sciences~\citep{bechtel1993, kuusela2012,
mclachlan2004finite}.   
The primary goal in many such applications  
is to perform statistical inference for the mixture parameters.
This raises the classical question of characterizing the optimal convergence rates
for parameter estimation in finite mixture models. Though this topic has been the subject of considerable investigation in past literature, 
the aim of our work is to show how these existing results may be refined through a careful choice of the loss function
used in their analyses.

Mixture distributions do not enjoy the standard regularity
conditions that are typically presumed in parametric models, such
as non-degeneracy of the Fisher information. 
As a result, optimal rates of estimation in mixtures are strictly slower than the usual parametric rate of convergence. This observation 
dates back at least to the seminal work of~\citet{chen1995}, who analyzed univariate mixtures
satisfying a regularity condition known as strong identifiability, which we formally define in Section~\ref{sec:background} below. 
A long line of recent work has further analyzed convergence rates in mixtures of general dimension, under  varying degrees of strong identifiability. 
In particular, \citet{nguyen2013} proposed the Wasserstein distance as a natural tool for metrizing convergence of parameters in finite mixtures, via
their mixing measure. The Wasserstein metric was then used 
to analyze convergence rates for the maximum likelihood estimator (MLE) and related procedures,
 under various classes of finite mixture models~\citep{ho2016EJS, ho2016convergence,heinrich2018, Ho_Nguyen_SIMODS}. 
Moment-based estimators were also studied by~\citet{wu2020a,doss2020}, and Bayesian estimators by~\citet{ohn2020,guha2021Bernoulli}, to name a few. 

A broad conclusion of these works is that slow convergence rates are pervasive to {\it parameter estimation}
in finite mixture models. This observation contrasts the fact that the minimax rate of estimating the {\it density} 
of a finite mixture model is typically the standard parametric rate of convergence~\citep{genovese2000,ghosal2001,doss2020,ashtiani2020}.  
For example, \citet{heinrich2018} show  that the minimax rate for parameter estimation in a strongly identifiable mixture 
degrades exponentially as the number of components increases, when no separation conditions are placed on these components. 
This result suggests that the estimation of mixture parameters can be prohibitive, even when the number of components 
is moderate.  On the other hand, practitioners have long been employing mixture models successfully, suggesting a discrepancy between
practice and the worst-case  rates suggested by the theory. 

The goal of this paper is to revisit existing convergence rates for parameter estimation in finite mixture models, and to show that they
may be refined by using stronger loss functions than the Wasserstein distance. 
We will argue that the Wasserstein distance is only able to capture the worst-case
convergence rate among the estimated components of a mixture, and that in many cases, the vast majority of estimated component parameters
may achieve considerably faster convergence rates than anticipated from prior work. 
Before describing these phenomena in further detail, we begin by formally introducing finite mixture models and related notions.

\subsection{Problem Setting}
{\bf Finite Mixture Models.} Let $\calF=\{f(x|\theta): 
x \in \calX, \theta \in \Theta \}$ be a known parametric family of density functions 
with respect to a dominating $\sigma$-finite measure $\nu$. 
Here, we assume $\calX \subseteq \bbR^N$ for some $N \geq 1$, 
and $\Theta$ is a parameter space which will either be a subset of the Euclidean space $\bbR^d$, $d \geq 1$, 
or of the set $\bbR^d \times \bbS^d_{++}$, where $\bbS^d_{++}$ denotes the cone of $d\times d$ positive
definite matrices. In either case, we shall always tacitly assume   that $\Theta$ is a  compact set with nonempty interior. 
Let $X_{1}, X_{2}, \ldots, X_{n}$ be an i.i.d. sample
from a finite mixture model with $k_{0}\geq 1$ components, whose density with respect
to $\nu$ is written as
$$p_{G_{0}}(x) : = \int f(x|\theta) dG_{0}(\theta) = \sum_{j = 1}^{k_{0}} p_{j}^{0} f(x|\theta_{j}^{0}), \quad x \in \calX.$$ 
Here $G_{0} = \sum_{j = 1}^{k_{0}} p_j^{0} \delta_{\theta_j^{0}}$ denotes an unknown mixing measure,
where the $p_j^0\geq 0$ are called mixing proportions (or weights), satisfying $\sum_{j=1}^{k_0} p_j^0 = 1$, and
the $\theta_j^0 \in \Theta$ are called atoms, for $j=1, \dots, k_0$. 
When the mixing proportions are strictly positive and the atoms are distinct, we say $G_0$ has true order $k_0$. 
More generally, any finitely-supported probability measure
on $\Theta$ is called a mixing measure, and
its support size is called its order.  The set of mixing measures of order at most $k \geq 1$ is denoted $\calO_k(\Theta)$,
and we write $\calE_k(\Theta) = \calO_k(\Theta) \setminus \calO_{k-1}(\Theta)$. 

When dealing with parameter estimation in a finite mixture model, it is convenient to treat the mixing measure $G_0$ as the target of estimation,
even if the main quantities of interest are the mixing proportions or atoms of $G_0$. Indeed, while the density $p_G$ is typically 
identifiable with respect to its mixing measure $G$, it is never identifiable with respect to the individual parameters of $G$, due to the possibility
of label-switching. 
Throughout our work, we will consider both {\it pointwise} rates of estimating  the mixing measure, that is, estimation rates
which depend on the fixed mixing measure $G_0$, and {\it uniform} estimation rates, which hold uniformly over
all mixing measures under consideration. We will always emphasize the latter setting by allowing $G_0 \equiv G_0^n$ to potentially
depend on the sample size $n$.

{\bf Maximum Likelihood Estimation.} Perhaps the most widely-used estimator of $G_0$ is the maximum likelihood estimator (MLE). 
We focus our analysis on  estimators based on the MLE throughout this work, in part because
they allow for a general theory of parameter estimation
to be derived under minimal conditions
on the family $\calF$. Given an integer $k \geq 1$, 
the MLE of $G_0$ with order at most $k$ is given by
\begin{equation}
\label{eq:mle}
\begin{aligned} 
\widebar G_n &= \sum_{i=1}^{\widebar k_n} \widebar p_i^n \delta_{\widebar\theta_i^n} = \argmax_{G \in \calO_k(\Theta)} \ell_n(G), \qquad\text{where } \ell_n(G) = \sum_{i=1}^n \log p_G(X_i).
\end{aligned} 
\end{equation}
Here, $\bar k_n\leq k$ denotes the fitted order of $\widebar G_n$. 
We have defined the MLE with the general order $k$ to reflect the fact that true order $k_0$ of $G_0$
may be unknown. Notice that $\widebar G_n$ is generally inconsistent if $k < k_0$, 
thus we shall always assume $k \geq k_0$. 
Our convergence rates will depend on the level of misspecification $k-k_0$.

In certain parts of our development, it will be technically convenient to ensure that the fitted mixing proportions of $\widebar G_n$
do not vanish. While this can be achieved by constraining the maximum in equation~\eqref{eq:mle}, 
we will prefer to achieve
this using a penalty on the likelihood function. Specifically, we follow~\citet{chen1996} and define the penalized MLE of order at most $k$ by 
$$\hat G_n = \sum_{i=1}^{\hat k_n} \hat p_i^n \delta_{\htheta_i^n} = \argmax_{G \in \calO_k(\Theta)} \ell_n(G) + \xi_n \pen(G),$$
where $\hat k_n \leq k$ is the  order of $\widehat G_n$, $\xi_n \geq 0$ is a tuning parameter, and $\pen$ satisfies $\pen(G) \to -\infty$
as the smallest mixing weight of $G$ vanishes. 
For concreteness, 
we will use the~penalty
$\rho(G) = \sum_{j=1}^{k'} \log  p_j'$,
where $k'\leq k $ denotes the order of $G= \sum_{j=1}^{k'} p_j' \delta_{\theta_j'}$.
As discussed in Appendix~\ref{app:numerical}, with this choice of penalty, $\hat G_n$ may be 
numerically approximated using a simple modification of the EM~algorithm.

 In order to evaluate the risk of the estimators $\hat G_n$ and $\widebar G_n$, we will require loss functions defined over 
$\calO_k(\Theta)$. The most widely-used loss function appearing in past work is the Wasserstein distance, which we define next. 

{\bf Wasserstein Distances.} 
Let $k,k'\geq 1$, and set $G = \sum_{i=1}^k p_i \delta_{\theta_i} \in \calO_k(\Theta)$
and $G' = \sum_{j=1}^{k'} p_j'\delta_{\theta_j'} \in \calO_{k'}(\Theta)$. Denote by $\Pi(G, G')$ the set of joint
probability mass functions $\bq  = (q_{ij}: i \in [k],j \in [ k'])$
admitting marginal distributions equal to those of $G$ and $G'$, that is, 
$\sum_{i=1}^k q_{ij} = p_j'$ and $\sum_{j=1}^{k'} q_{ij} = p_i$, for all $i\in [k]$
and $j\in [k']$. The Wasserstein distance of order $r \geq 1$ is defined by 
$$W_r(G, G') = \left(\inf_{\bq \in \Pi(G, G')} \sum_{i=1}^k \sum_{j=1}^{k'} q_{ij} D^r(\theta_i, \theta_j')\right)^{\frac 1 r},$$
where $D$ is a metric on $\Theta$. When $\Theta \subseteq \bbR^d$, 
we shall always assume that $D = \|\cdot\|$ is induced by the Euclidean norm.  

The use of Wasserstein distances in general dimension originated
from the work of~\citet{nguyen2013}, and was partly motivated by its implication for the convergence of atoms, as we now recall. 
Let $G_n \in \calO_k(\Theta)$
be a sequence of mixing measures, and $G_0 \in \calE_{k_0}(\Theta)$. Then, if $W_r(G_n, G_0) \leq \alpha_n$ for some 
$\alpha_n \downarrow 0$, there exists a subsequence of $G_n$ such that every atom $\theta_{j}^0$ of $G_0$
is the limit point of at least one  atom $\theta_i^n$ of $G_n$. Furthermore, the convergence rate of 
this fitted atom is $D(\theta_i^n,\theta_{j}^0) \lesssim \alpha_n$. 
When $k > k_0$, there may also be atoms $\theta_\ell^n$ of $G_n$
which do not converge to any atoms of $G_0$. It can be seen that their corresponding mixing proportions $p_\ell^n$ must
then vanish at the rate $\alpha_n^r$. If we instead assume that the mixing proportions of $G_n$ are bounded from below by a positive
constant $c_0 > 0$, it must in fact hold that every atom of $G_n$  converges to an atom of $G_0$ at rate $\alpha_n$. 

We note in particular that the Wasserstein distance can only induce the same convergence rate $\alpha_n $ for those atoms of $G_n$
which approach the atoms of $G_0$. 
In contrast, a key observation of our work is that maximum likelihood-based estimators have atoms which converge at distinct rates; such heterogeneous behaviour cannot be captured by the Wasserstein distance, and is the main subject of this paper.

\subsection{Contributions}
Our goal
is to provide sharper rates of convergence for parameter
estimation in finite mixture models of various types. 
Our main technical contribution is the development of loss
functions over the space of mixing measures, 
which are stronger than
the Wasserstein distance, and which 
correctly characterize the heterogeneous
convergence rates of the various mixture parameters  in 
maximum likelihood-based estimators.
To illustrate the refinements furnished by our theory, 
we consider the following example.
\begin{example}[Pointwise Convergence Rates for Strongly Identifiable Mixtures]
Suppose $\calF$ 
is the location family of Gaussian densities
with known variance. Furthermore, assume $k=k_0+1$. The works of~\citet{chen1995,ho2016EJS} show there exists a constant $C(G_0) > 0$ such that
$$\bbE W_2(\widebar G_n, G_0) \leq C(G_0) (\log n/n)^{1/4}.$$
In particular, it follows that for every atom 
$\theta_j^0$ of $G_0$, there is at least one atom of $\widebar G_n$ which
converges to $\theta_j^0$ at the pointwise rate $(\log n/n)^{1/4}$.
Equivalently, there exists
an injection $u_n : [k_0] \to [k]$ such that 
\begin{equation} 
\label{eq:example_past_rate}
\max_{1 \leq j \leq k_0} \bbE \|\widebar \theta_{u_n(j)}^n - \theta_{j}^0\| \leq C(G_0) \left(\log n/n\right)^{\frac 1 4}.
\end{equation}
In contrast, it will follow from our Theorem~\ref{theorem:rate_MLE_strong_identifiability}
below that 
there exists an injection $v_n:[k_0] \to [k]$ and a permutation $\sigma_n:[k_0]\to[k_0]$ such that
\begin{align*}
\max_{1 \leq j \leq k_0-1} \bbE \|\widebar \theta_{v_n(j)}^n - \theta_{\sigma_n(j)}^0\| &\leq C(G_0) \left(\frac{\log n}{n}\right)^{\frac 1 2}, \\
\bbE \|\widebar \theta_{v_n(k_0)}^n - \theta_{\sigma_n(k_0)}^0\| &\leq C(G_0) \left(\frac{\log n}{n}\right)^{\frac 1 4}.
\end{align*}
This result shows that, ignoring
polylogarithmic factors, all but two
of the atoms of the overfitted MLE $\widebar G_n$ 
achieve the parametric convergence rate. In contrast, equation~\eqref{eq:example_past_rate} merely
shows that  these atoms
converge at the slower rate $(\log n/n)^{1/4}$.
\end{example}

We will show that similar asymptotics hold for a broad
family of strongly identifiable mixture models, and for general
$k \geq k_0$, 
in Section~\ref{sec:pointwise_strongly_identifiable},
We further consider uniform convergence 
rates for such families
in Section~\ref{sec:uniform_settings},
as well as 
pointwise convergence rates
for location-scale
Gaussian mixture models
(Section~\ref{sec:pointwise_weakly_identifiable}), which form an important example
of weakly identifiable finite mixtures. 
We obtain these results
by identifying distinct loss functions
tailored to each of these three settings,
which accurately capture the behaviour of
individual fitted mixture parameters. 

Our results highlight the underappreciated
fact that the Wasserstein distance 
merely quantifies the worst-case
convergence rate among the fitted parameters
of a finite mixture; 
its use in past work may thus have painted an overly pessimistic picture of parameter estimation 
in these models. Though our primary emphasis is
on such theoretical aspects, we 
will also discuss
that certain loss functions developed
in this work enjoy an improved
computational complexity as compared to the Wasserstein distance, and may
therefore be of practical significance
in their own right. 

\noindent

\textbf{Notation.} Given  probability densities $p,q$   dominated by $\nu$, their squared
 Hellinger and Total Variation distances
 are denoted by 
 $h^2(p,q) = \frac 1 2 \int (\sqrt p - \sqrt q)^2 d\nu$
 and $V(p,q) = \frac 1 2 \int |p-q| d\nu$.    $\calO_{k,c_0}(\Theta)$ denotes the set
of mixing measures in $\calO_k(\Theta)$ with mixing weights bounded below by a constant $c_0 >0$, 
and  $\calE_{k,c_0}(\Theta) =  \calO_{k,c_0}(\Theta)\setminus\calO_{k-1}(\Theta)$. 
For any $n \geq 1$, we denote $[n] = \{1, 2, \ldots, n\}$.
For any $a,b \in \bbR$, $a\vee b = \max\{a,b\}$ and
$a\wedge  b = \min\{a,b\}$. 
 Given  $(a_n)_{n\geq 1}, (b_n)_{n\geq 1} \subseteq \bbR_+$, we write $a_n \lesssim b_n$ if
 there exists a universal constant $C > 0$, possibly
 depending on   problem parameters to be understood
 from context, such that $a_n \leq C b_n$ for all $ n\geq 1$. 
We also write $a_n\asymp b_n$ when  $a_n \lesssim b_n \lesssim a_n$. 
$\calC^\alpha(\Theta)$
 denotes the H\"older space of regularity
 $\alpha > 0$ over $\Theta$, 
 with associated norm $\|\cdot\|_{\calC^\alpha(\Theta)}$~\citep{folland1995}.

\section{Preliminaries}
\label{sec:background}
\subsection{Strong Identifiability}
We begin by recalling
the  strong identifiability condition 
for the parametric family $\calF$.
\begin{definition}[Strong Identifiability]
Let $r \geq 0$ be an integer. We say $\calF$
is $r$-strongly identifiable
if $f(x|\cdot) \in \calC^r(\Theta)$ 
for $\nu$-almost every $x \in \calX$, 
and if for any $k \geq 1$,
and $\theta_1,\dots,\theta_k \in \Theta$, 
the following
implication holds for all
$\alpha_\eta^{(i)} \in \bbR$, 
\begin{align*}
\esssup_{x\in \calX} \Bigg| \sum_{\ell=0}^r \sum_{|\eta|=\ell} 
\sum_{i=1}^k  \alpha_\eta^{(i)} \frac{\partial^{|\eta|}
f}{\partial\theta^\eta}(x|\theta_i)\Bigg| = 0
\Longrightarrow 
\max_{|\eta|\leq r} \max_{1 \leq i \leq k} |\alpha_\eta^{(i)}| = 0.
\end{align*}
\end{definition}
The notion of strong
identifiability originates from the work
of~\citet{chen1995}, and is stated here
in a more general form due to~\citet{heinrich2018,ho2016EJS}. 
We refer to these references, as well as that
of~\citet{holzmann2004}, for sufficient
conditions under which the strong
identifiability condition holds. 
For example, this condition is known
to be satisfied for any finite $r \geq 1$
by the location Gaussian parametric family
with known scale parameter, the Poisson
family, and other common exponential families. 
Location-scale Gaussian densities form perhaps
the most widely-used parametric family
which fails to satisfy the $r$-strong identifiability
condition for $r \geq 2$~\citet{ho2016convergence}, 
and we will treat this special case separately. 

We will typically
couple the strong identifiability
condition with the following
assumption on the modulus
of continuity of the derivatives
of $f(x|\cdot)$, up to order $r \geq 1$. 
\begin{enumerate}[leftmargin=1.05cm,listparindent=-\leftmargin,label=\textbf{A($r$)}]   
\item \label{assm:holder}
There exist $\Lambda,\delta > 0$ such that
$$\esssup_{x \in \calX}  
\|f(x|\cdot)\|_{\calC^{r+\delta}(\Theta)} \leq 
\Lambda.$$
\end{enumerate}
Strong identifiability generalizes the condition
of regular identifiability of the family
$\calP_k(\Theta) = \{p_G:G \in \calO_k(\Theta)\}$, and 
is a useful notion for deriving
inequalities between 
Wasserstein-type distances
over $\calO_k(\Theta)$ and statistical distances over
$\calP_k(\Theta)$. Such bounds
are at the heart of our proofs, and
will allow us to derive parameter estimation rates
from known
convergence rates for maximum likelihood density estimation, 
to which we turn our attention~next. 

\subsection{Convergence Rates for Maximum Likelihood
Density Estimators}
In order to state a rate of convergence
for the density estimators
$p_{\hat G_n}$ and $p_{\widebar G_n}$, for instance
under the Hellinger distance, 
we require a condition on the complexity
of the class
\begin{equation*} 
\widebar\calP_k^{1/2}(\Theta,\epsilon)  
 = \left\{\bar p_G^{1/2}  : G \in \calO_k(\Theta), ~ h(\bar p_G, p_{G_0}) \leq \epsilon\right\},
\end{equation*}
where $\epsilon > 0$, and for any $G \in \calO_k(\Theta)$, 
we write $p_G = (p_G + p_{G_0})/2$. 
The definition of $\widebar\calP_k^{1/2}(\Theta,\epsilon)$
originates from~\citet{vandegeer2000}, 
who place conditions on the convex
combinations $\bar p_G$, rather
than $p_G$, as this choice
is guaranteed to place a
non-negligible amount
of probability mass
over the support of $p_{G_0}$. 
The complexity of this class
is measured 
through the bracketing entropy integral
\begin{equation*}   
\begin{aligned}
\calJ_B&(\epsilon, \widebar\calP_k^{1/2}(\Theta, \epsilon), \nu)
 {=} \int_0^\epsilon \sqrt{H_B(u, \widebar\calP_k^{1/2}(\Theta,u), \nu)}du~ {\vee} \epsilon,
 \end{aligned}
 \end{equation*}
 where $H_B(\epsilon, \calP, \nu)$ denotes
 the $\epsilon$-bracketing entropy
 of a set $\calP \subseteq L^2(\nu)$
 with respect to the $L^2(\nu)$ metric~\citep{vandegeer2000}. 
We shall assume that this quantity
satisfies the following condition. 
\begin{enumerate}[leftmargin=1.05cm,listparindent=-\leftmargin,label=\textbf{B($k$)}]   
\item  \label{assm:bracket}
Given a universal constant $J > 0$,
there exists a constant $L > 0$, possibly
depending on $d$ and $k$, such that
for all $n \geq 1$ and all $\epsilon > L (\log n/n)^{1/2}$,
$$\calJ_B(\epsilon, \widebar\calP_k^{1/2}(\Theta,\epsilon), \nu) \leq J \sqrt n \epsilon^2.$$
\end{enumerate}
We are now ready to state
the following convergence rates.
\begin{theorem}
\label{thm:density_estimation_rate_penalized}
Given $k \geq 1$, 
assume  condition~\ref{assm:bracket} holds.
\begin{enumerate} 
\item[(i)] There exists a  constant $C > 0$ depending
only on $d,k,\calF$ such that for all $n \geq 1$, 
$$\sup_{G_0 \in \calO_{k}(\Theta)}
\bbE_{G_0} h(p_{\widebar G_n}, p_{G_0})  
\leq  C\sqrt{\frac{\log n}{n}}.$$
\item[(ii)]
Furthermore, given $c_0,c_1 > 0$, if $0 \leq \xi_n \leq c_1 \log n$, then there exists a constant $C' > 0$ depending on
$d,k,c_0,c_1,\calF$ such that for all $n \geq 1$, 
$$\sup_{G_0 \in \calO_{k,c_0}(\Theta)}
\bbE_{G_0} h(p_{\hat G_n}, p_{G_0})  
\leq  \frac{C'\log n}{\sqrt n}.$$
\end{enumerate}
\end{theorem}
Theorem~\ref{thm:density_estimation_rate_penalized}(i) is a direct
consequence of generic
results for maximum 
likelihood density estimation
(for instance, Theorem 7.4 of~\citet{vandegeer2000}). 
Its application to finite
mixture models has previously
been discussed by~\citet{ho2016EJS},
who also argue that condition~\ref{assm:bracket}
is satisfied by a broad
collection of parametric families $\calF$, 
including the multivariate location-scale Gaussian 
and Student-$t$ families. 
A version of 
Theorem~\ref{thm:density_estimation_rate_penalized}(ii)
is implicit in the work of~\citet{manole2021b}, 
though with a stronger condition
on the tuning parameter $\xi_n$. We provide
a self-contained proof of this result
in Appendix~\ref{sec:proof} for completeness. 

These results  
may also be used to show that the penalized
MLE has nonvanishing
mixing proportions. 
\begin{proposition}
\label{prop:bded_mixing}
Let $k \geq 1$, $c_0 \in (0,1)$, and assume
condition~\ref{assm:bracket} holds. 
Assume further that $\xi_n \geq \log n$. 
Then, there exists a constant $c > 1$
depending on $c_0, d, k,\calF$ such that
for all $n \geq 1$, 
$$\sup_{G_0 \in \calO_{k,c_0}(\Theta)}
\bbP_{G_0}\left( \min_{1 \leq j \leq {\hat k_n}}\hat p_j^n 
\geq \frac 1 c \right) \leq \frac c n.$$
\end{proposition}
In view of Proposition~\ref{prop:bded_mixing}
and Theorem~\ref{thm:density_estimation_rate_penalized}, 
we shall always tacitly assume that the tuning parameter
 $\xi_n$ is equal to~$\log n$.
 
\section{Pointwise Convergence Rates of the MLE}
\label{sec:pointwise_settings}
We first derive pointwise convergence
rates for estimating a fixed  mixing measure $G_{0} \in \calE_{k_0}(\Theta)$.
\subsection{Strongly Identifiable Case} 
\label{sec:pointwise_strongly_identifiable}
Assume the family $\calF$ is twice strongly
identifiable, with 
a compact parameter space
$\Theta \subseteq \bbR^d$ admitting nonempty interior.
We begin by defining a loss function on $\calO_k(\Theta)$
tailored to this setting. 
Given a mixing measure
$G = \sum_{i = 1}^{k'} p_{i} \delta_{\theta_{i}}$ of order $k' \leq k$, we partition its atoms
into the following Voronoi cells, 
generated by the support
of $G_{0}$,  
$$\mathcal{A}_{j} \equiv \mathcal{A}_{j}(G) = 
  \{i \in [k']: \|\theta_{i} - \theta_{j}^{0}\| \leq \|\theta_{i} - \theta_{\ell}^{0}\| \ \forall \ell \neq j
  \},$$ 
for all $j \in [k_{0}]$. We may then define the loss function
\begin{align}
    \mathcal{D}(G, G_{0}) : = \sum_{j: |\mathcal{A}_{j}| > 1} \sum_{i \in \mathcal{A}_{j}} p_{i} \|\theta_{i} - \theta_{j}^{0}\|^2  + \sum_{j: |\mathcal{A}_{j}| = 1} \sum_{i \in \mathcal{A}_{j}} p_{i} \|\theta_{i} - \theta_{j}^{0}\| + \sum_{j = 1}^{k_{0}} \left|\sum_{i \in \mathcal{A}_{j}} p_{i} - p_{j}^{0}\right|. \label{eq:first_discrepancy} 
\end{align}
Clearly,  $\mathcal{D}(G, G_{0}) {=} 0$ if and only if $G {=} G_{0}$. Under this~loss function, we obtain the following bound on the risk of~$\widebar G_n$. 
\begin{theorem}
\label{theorem:rate_MLE_strong_identifiability}
Let $k \geq k_0$. Assume that the parametric  family $\calF$ is 2-strongly identifiable, and satisfies conditions~\hyperref[assm:holder]{\textbf{A($2$)}} and~\ref{assm:bracket}. Then,  
there exists a constant $C(G_0) > 0$, 
depending on $G_0,d,k,\calF$, such that
\begin{align*}
    \bbE \big[\mathcal{D}(\widebar{G}_{n}, G_{0}) \big]
    \leq C(G_0) \sqrt{ \frac{\log n}{n}}.
\end{align*}
\end{theorem}
The proof of Theorem~\ref{theorem:rate_MLE_strong_identifiability} appears in Appendix~\ref{subsec:proof:theorem:rate_MLE_strong_identifiability}, where the main 
difficulty is to prove the following lower bound of the Hellinger distance in terms of $\calD$,
\begin{align}
   \calD(G,G_{0}) \leq C(G_0) h(p_{G}, p_{G_{0}}),
\end{align}
for any $G \in \mathcal{O}_{k}(\Theta)$. 
Using Theorem~\ref{thm:density_estimation_rate_penalized}(i), the above bound directly leads to the
stated convergence rate of $\widebar{G}_{n}$.

A few comments regarding Theorem~\ref{theorem:rate_MLE_strong_identifiability} are in order. First,
let $\mathcal{A}_{j}^{n} = \mathcal{A}_{j}(\widebar{G}_{n})$ for all $j \in [k_{0}]$.
The convergence rate $\sqrt{\log n/ n}$ of $\mathcal{D}(\widebar{G}_{n},G_{0})$ implies that for any index $j \in [k_{0}]$ such that $|\mathcal{A}_{j}^{n}| = 1$, $\| \widebar{\theta}_{i}^{n} - \theta_{j}^{0}\|$ and $|\widebar{p}_{i}^{n} - p_{i}^{0}|$ vanish at the near-parametric rate $\sqrt{\log n/ n}$ for $i \in \mathcal{A}_{j}^{n}$. 
Therefore, among the true components
which are only approximated by a single fitted component,
the parameters of this fitted component
converge as fast as if the
order $k \geq k_0$ were not overspecifed.
In particular, in the exact-fitted setting $k = k_0$, 
we find that all fitted components
and mixing proportions
converge at the parametric rate, up 
to a polylogarithmic factor, which
recovers Theorem~3.1 of~\citet{ho2016EJS}. 
Furthermore, when $k > k_0$, 
for any index $j \in [k_{0}]$ such that $|\mathcal{A}_{j}^{n}| \geq 2$, $\sum_{i \in \mathcal{A}_{j}^{n}} \widebar{p}_{i}^{n} \|\widebar{\theta}_{i}^{n} - \theta_{j}^{0}\|^2$ and $|\sum_{i \in \mathcal{A}_{j}^{n}} \widebar{p}_{i}^{n} - p_{j}^{0}|$ decay at the rate $\sqrt{\log n/ n}$. 
In particular, it follows that for
every such $j$, there exists $i \in \mathcal{A}_{j}^{n}$ such that $\widebar{\theta}_{i}^{n}$ converges to $\theta_{j}^{0}$ at the rate $(\log n/ n)^{1/4}$,
which is now markedly slower than the parametric
rate. In contrast, 
the past works of~\citet{chen1995, nguyen2013, ho2016EJS} show that $\bbE W_{2}^2(\widebar{G}_{n}, G_{0}) \lesssim \sqrt{\log n/n}$, which 
implies a convergence rate no better 
than $(\log n/ n)^{1/4}$ 
for all atoms of the MLE, rather than just
those lying in a set~$\mathcal{A}_{j}^{n}$ with cardinality greater than one. 
These existing results painted a pessimistic picture
of maximum likelihood estimation in 
overspecified mixtures---for example, they suggest
that overspecifying the order $k_0$
merely by $k=k_0+1$ leads to poor convergence 
rates for each of the $k$ fitted atoms, whereas our
work shows that 
at least $k_0-1$ fitted atoms enjoy
considerably faster convergence rates.  


Second, we can demonstrate that $\mathcal{D}  \gtrsim W_{2}^2 $, and 
$$\sup_{\substack{G \neq G_0 \\ G \in \mathcal{O}_{k}(\Theta)}} \mathcal{D}(G, G_{0})/ W_{2}^2(G, G_{0}) =\infty.$$
See Lemma~\ref{lemma:relation_strong_identifiability} in Appendix~\ref{sec:auxiliary} for a formal statement. 
This shows that $\mathcal{D}$ is a stronger loss function than the Wasserstein distance.
In particular, we deduce that 
that Theorem~\ref{theorem:rate_MLE_strong_identifiability} 
also implies the 
aforementioned convergence rate of $\widebar{G}_{n}$ under the Wasserstein distance.  

Finally, the complexity of computing $\mathcal{D}(G, G_{0})$
is of the order of $O(k \times k_{0})$. In contrast, computing $W_{2}(G, G_{0})$ is equivalent to solving a linear programming problem, which has complexity no better than $O(k^3)$~\citep{pele2009}. Therefore, the loss function $\mathcal{D}$ is computationally more efficient than the Wasserstein metric. 
This observation is significant because
the Wasserstein distance has previously
been used 
as a methodological tool for model
selection 
in finite mixtures~\citep{guha2021Bernoulli}. In these applications, 
the loss function $\calD$  provides
an alternative to $W_2^2$ which is both statistically
and computationally more~efficient. 
\begin{figure*}
    \centering
    \begin{tabular}{cc}
        \includegraphics[width=80mm]{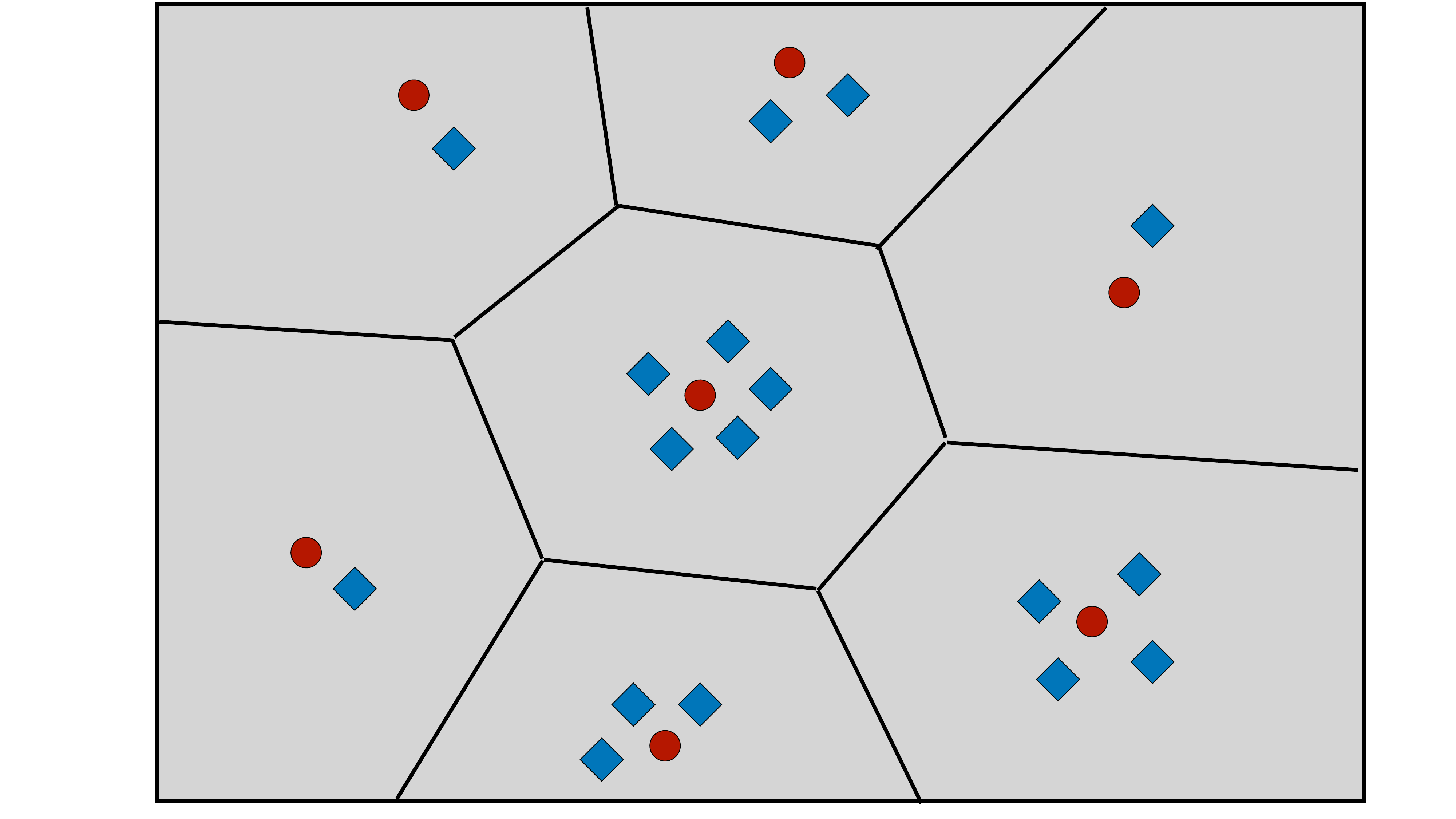} &
        \includegraphics[width=80mm]{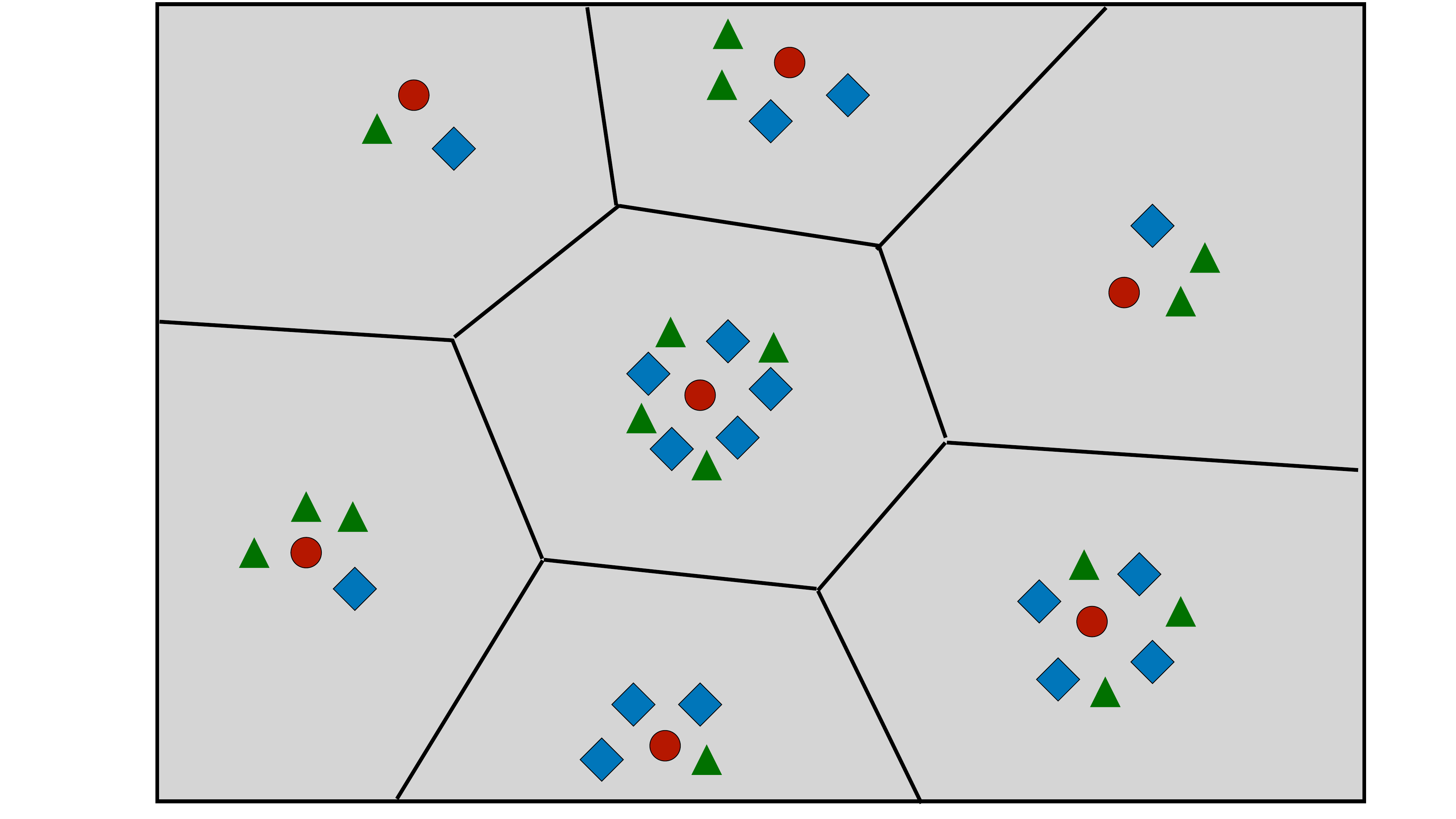} \\
        (a) & (b)
    \end{tabular}
    \caption{\footnotesize{(a) Illustration of the Voronoi cells generated by the atoms of the true mixing measure $G_{0}$ (red  points), and of the  convergence rates of the fitted atoms of the (possibly penalized) MLE (blue  points), under the pointwise setting. The cardinality of each Voronoi cell
    is the number of atoms of the MLE in these cells. The atoms and mixing weights of the MLE in the Voronoi cells with cardinality one have $n^{-1/2}$ convergence rates,
    where we ignore polylogarithmic factors. When the model is 2-strongly identifiable, the atoms of the MLE in the Voronoi cells with cardinality greater than one converge at the slow rate $n^{-1/4}$, while their mixing weights have $n^{-1/2}$ rates of convergence. Under location-scale Gaussian mixtures, the  location and scale mixing components of the Voronoi cells with $l \geq 2$ elements respectively have convergence rates $n^{-1/2\bar{r}(l)}$ and $n^{-1/\bar{r}(l)}$ while their mixing weights have $n^{-1/2}$ rates of convergence. (b) Illustration of the Voronoi cells generated by the limiting mixing measure $G_{*}$ under the uniform setting of Section~\ref{sec:uniform_settings}.
    The red, blue, and green points respectively
    denote the atoms of the limiting measure $G_*$, 
    the penalized MLE $\widehat G_n$, and
    the varying true mixing measure $G_{0}^{n}$. 
    The atoms in each Voronoi cell with $l \geq 2$ atoms
    of $\widehat G_n$ or $G_0^n$   converge at the rate $n^{-1/2(l - 1)}$.}    \label{fig:illustration_rate}} 
\end{figure*}

\subsection{Weakly Identifiable Case: Location-Scale Gaussian Mixtures} 
\label{sec:pointwise_weakly_identifiable}
In this section, we study the convergence rate of the MLE when the model is not strongly identifiable in the second order. Location-scale Gaussian mixtures 
are a popular example of such models, as
a result of the following equation:
\begin{align}
    \frac{\partial^2{f}}{\partial{\mu}\partial\mu^\top}(x| \mu, \Sigma) = 2 \frac{\partial{f}}{\partial{\Sigma}}(x|\mu, \Sigma), \label{eq:Pde_Gaussian}
\end{align}
for all $x \in \bbR^d$ and $\theta=(\mu,\Sigma) \in\Theta$, where $\calF=\{f(\cdot|\theta):\theta\in\Theta\}$ 
denotes the family of location-scale Gaussian densities, with compact parameter space $\Theta \subseteq \bbR^d \times \bbS^{d-1}$. The absence of second order identifiability in location-scale Gaussian mixtures leads to several challenges in studying the convergence rates of the MLE. To simplify our proofs, we will assume that 
all mixing measures have weights
which are lower bounded by some  small  constant $c_{0} > 0$. As a result, 
we only state a convergence rate
for the penalized MLE $\widehat G_n$, 
which indeed lies in the class $\calO_{k,c_0}(\Theta)$
with high probability, by Proposition~\ref{prop:bded_mixing}. 
We would like to remark that constraints on the mixing weights are also assumed in past work 
on convergence rates
for over-specified location-scale Gaussian mixtures~\citep{ho2016convergence}, and are not
a byproduct of our choice of loss function.  

Proposition 2.2 in~\citet{ho2016convergence},
together with Theorem~\ref{thm:density_estimation_rate_penalized}
and Proposition~\ref{prop:bded_mixing},
may be used to establish the following bound, for some constant
$C(G_0) > 0$,
\begin{align*}
\bbE \big[W_{\overline{r}(k - k_{0}+1)} (\widehat{G}_{n}, G_{0})\big] \leq C(G_0) \left(\frac{\log n}{\sqrt n}\right)^{\frac 1 {\overline{r}(k - k_{0} + 1)}},
\end{align*}
where for any $k' \geq 2$, $\overline{r}(k')$ is 
defined as the smallest integer $r$ such that the system of polynomial equations
\begin{eqnarray}
\sum \limits_{j=1}^{k'} \sum \limits_{n_{1}, n_{2}} \dfrac{c_{j}^{2}a_{j}^{n_{1}}b_{j}^{n_{2}}}{n_{1}!n_{2}!} = 0,~~ \ \text{for each} \ \alpha=1,\ldots,r \label{eqn:system_polynomial_Gaussian_first}
\end{eqnarray}
does not have any nontrivial solution for the unknown variables $(a_{j},b_{j},c_{j})_{j=1}^{k'}\subseteq\bbR$. The range of $(n_{1},n_{2})$ in the second sum consist of all natural pairs satisfying the equation $n_{1}+2n_{2}=\alpha$. A solution to the above system is considered nontrivial if all variables $c_{j}$ are non-zero, while at least one of the $a_{j}$ is non-zero.
For example, it was shown by~\citet{ho2016EJS}
that $\rbar(2) = 4$ and $\rbar(3) = 6$. 

The convergence rate $(\log n/\sqrt  n)^{1/\overline{r}(k - k_{0} + 1)}$ of $\widehat G_n$  indicates that the location and scale parameters of the penalized MLE converge to
their population counterparts at this same slow rate.
As before, this result does not precisely reflect the  behavior of individual  parameters in  location-scale Gaussian mixtures, leading us to consider
a stronger loss function than the Wasserstein distance.
Given $G = \sum_{i = 1}^{k'} p_{i} \delta_{(\mu_{i}, \Sigma_{i})} \in \calE_{k'}(\Theta)$ for $k' \leq k$,
define the Voronoi cells $\mathcal{A}_{j} = \mathcal{A}_{j}(G) = \{i \in [k']: \|\mu_{i} - \mu_{j}^{0}\| + \| \Sigma_{i} - \Sigma_{j}^{0}\|  \leq \|\mu_{i} - \mu_{\ell}^{0}\| + \| \Sigma_{i} - \Sigma_{\ell}^{0}\| \ \forall \ell \neq j \}$,
for $j \in [k_0]$, and set
\begin{align}
     \widebar{\mathcal{D}}(G, G_{0}) \nonumber 
    & := \sum_{j: |\mathcal{A}_{j}| = 1} \sum_{i \in \mathcal{A}_{j}} p_{i} \left(\|\mu_{i} - \mu_{j}^{0}\| + \|\Sigma_{i} - \Sigma_{j}^{0}\| \right) \\
    &+ \sum_{j: |\mathcal{A}_{j}| > 1} \sum_{i \in \mathcal{A}_{j}} p_{i} \left( \|\mu_{i} - \mu_{j}^{0}\|^{\bar{r}(|\mathcal{A}_{j}|)} + \|\Sigma_{i} - \Sigma_j^{0}\|^{\frac{\rbar(|\calA_j|)}{2}}
    \right)\nonumber  
    \nonumber + \sum_{j = 1}^{k_{0}} \left|\sum_{i \in \mathcal{A}_{j}} p_{i} - p_{j}^{0}\right|.
\end{align}
It can be shown that $\widebar{\mathcal{D}} \gtrsim W_{\overline{r}(k - k_{0}+1)}^{{\widebar{r}(k - k_{0}+1)}}$ and 
$$\sup_{\substack{G \neq G_0
\\G \in \mathcal{O}_{k}(\Theta) }} \widebar{\mathcal{D}}(G, G_{0})/ W_{\overline{r}(k - k_{0}+1)}^{{\widebar{r}(k - k_{0}+1)}}(G, G_{0}) = \infty.$$ The proof is similar to that of Lemma~\ref{lemma:relation_strong_identifiability} in Appendix~\ref{sec:auxiliary}; therefore, it is omitted. We deduce that
$\widebar{\mathcal{D}}$ is a stronger loss function than $W_{\overline{r}(k - k_{0}+1)}^{{\overline{r}(k - k_{0}+1)}}$. 
We bound the risk of the penalized MLE under
$\widebar\calD$ as follows.
\begin{theorem}
\label{theorem:rate_MLE_Gaussian}
Let $\calF$ denote the location-scale Gaussian
density family with parameter
space taking the form $\Theta = [-a,a]^d \times\Omega$,
where $a > 0$ and $\Omega$
is a compact subset of $~\bbS^{d-1}$  whose
eigenvalues lie in a closed interval 
contained in $(0,\infty)$. Then, there exists a constant $C(G_0) > 0$, depending only on $G_0,k,d,\Theta$, such that
\begin{align*}
\bbE\left[ \widebar{\mathcal{D}}(\widehat{G}_{n},G_{0})\right] 
\leq C(G_0) \frac{\log n}{\sqrt n}.
\end{align*}
\end{theorem}
\noindent 
The proof of Theorem~\ref{theorem:rate_MLE_Gaussian}
appears in Appendix~\ref{subsec:proof:theorem:rate_MLE_Gaussian}. Recall that $\widehat{G}_{n} = \sum_{i = 1}^{\hat k_{n}} \widehat{p}_{i}^{n} \delta_{(\widehat{\mu}_{i}^{n}, \widehat{\Sigma}_{i}^{n})}$, and write $\mathcal{A}_{j}^{n} = \mathcal{A}_{j}(\widehat{G}_{n})$ for all $j \in[k_0]$. 
Theorem~\ref{theorem:rate_MLE_Gaussian} 
implies the following.

\begin{enumerate} 
\item[(i)] Given $j \in [k_{0}]$ such 
that $|\mathcal{A}_{j}^{n}| \geq 2$,
we have, with probability tending to one,
\begin{align*} 
\| \widehat{\mu}_{i}^{n} - \mu_{j}^{0}\| &\lesssim (\log n/ \sqrt n)^{1/\bar{r}(|\mathcal{A}_{j}^{n}|)}, \quad \text{and,}
\quad \|\widehat{\Sigma}_{i}^{n} - \Sigma_{j}^{0}\|  \lesssim (\log n/\sqrt  n)^{2/\bar{r}(|\mathcal{A}_{j}^{n}|)},~~  i \in \calA_j^n.
\end{align*}
In particular, the location parameters 
of $\widehat G_n$ converge quadratically
slower than the scale parameters.
 
\item[(ii)] On the other hand, for any index $j \in [k_{0}]$ such that $|\mathcal{A}_{j}^{n}| = 1$ and for any $i \in \mathcal{A}_{j}^{n}$, we have with probability tending to one,
\begin{align} 
\label{eq:fast_rate_location_scale} 
\| \widehat{\mu}_{i}^{n} - \mu_{j}^{0}\| 
\vee \|\widehat{\Sigma}_{i}^{n} - \Sigma_{j}^{0}\| \lesssim \log n/ \sqrt n.
\end{align}
Hence, both location and scale
parameters of $\widehat G_n$
achieve the standard parametric rate
up to a logarithmic factor. We refer to Figure~\ref{fig:illustration_rate}(a) for an illustration of these convergence rates.
 
\item[(iii)] Notice that
$|\mathcal{A}_{j}^{n}| \leq \hat k_{n} - k_{0} + 1$
for all $j \in [k_0]$. When equality is achieved
for some $j$, there must
be a single Voronoi cell
with $\hat k_{n} - k_{0} + 1$ elements, while the remaining cells each
have exactly one component. 
In this case, there are $k_{0} - 1$ components of the penalized MLE which achieve the fast pointwise rate~\eqref{eq:fast_rate_location_scale}. 

\item[(iv)] When $k = k_{0} + 1$, there exists a unique index $j$ such that $\mathcal{A}_j^{n}$ has at most two components, while the remaining Voronoi cells have exactly one component. Since $\overline{r}(2) = 4$, this demonstrates that 
the two components having indices in $\mathcal{A}_{j}$ 
have means converging at the slow rate $n^{-1/8}$,
and covariances converging at the rate $n^{-1/4}$, up to polylogarithmic factors.
These particular
rates were already anticipated
by the work of~\citet{chen2003} 
when $k_0 = 1$. 
When $k_0 > 1$, our work shows that the remaining $k_0-1$ atoms of the penalized MLE converge at the fast rate~\eqref{eq:fast_rate_location_scale}. 

\item[(v)] When $k = k_{0} + 2$, there are two possible cases: 
either (a)
there exists a unique index $j'$ such that $\mathcal{A}_{j'}^{n}$ has at most three components while the remaining sets have exactly one component, or (b) 
 there exist indices $j'_{1}$ and $j_{2}'$ such that $\mathcal{A}_{j_{1}'}^{n}$ and $\mathcal{A}_{j_{2}'}^{n}$ have at most two components while the remaining sets have exactly one component. Under case (a), since $\overline{r}(3) = 6$, 
the means with indices in $\mathcal{A}_{j'}^{n}$ 
converge at the rate $(\log n/ n)^{1/12}$ while the remaining atoms of $\widehat{G}_{n}$ converge at the parametric rate. Under case (b), the means with indices in $\mathcal{A}_{j_{1}'}^{n}\cup\mathcal{A}_{j_{2}'}^{n}$   converge at the $(\log n/ n)^{1/8}$ rate while the remaining atoms converge at the  rate $(\log n/ n)^{1/2}$.
\end{enumerate}
Finally, similarly to the loss function $\mathcal{D}$ in equation~\eqref{eq:first_discrepancy}, 
we note that $\widebar{\mathcal{D}}(G,G_0)$ 
can be computed in $O(k \times k_{0})$ time
for any given
$G \in \calO_k(\Theta)$, 
and thus enjoys a computational advantage
over the Wasserstein metric.

\section{Uniform Convergence Rates of the MLE}
\label{sec:uniform_settings}
Thus far, we have derived pointwise convergence 
rates for the MLE or penalized MLE, 
which depend on the fixed mixing measure $G_0$. 
We next consider uniform rates of convergence, 
in which we allow the true mixing measure $G_{0}\equiv G_0^n \in \calE_{k_0}(\Theta)$ to
vary with the sample size~$n$, while converging
to some limiting mixing measure $G_{*} = \sum_{i = 1}^{k_{*}} p_{i}^{*} \delta_{\theta_{i}^{*}}\in \calE_{k_*}(\Theta)$,
of order $k_{*} \leq k_{0} \leq k$.
To simplify our proofs, 
we will assume throughout this section that~$\Theta\subseteq \mathbb{R}$.

It is known that the optimal pointwise rate of 
estimation in a strongly identifiable
mixture differs from the optimal uniform rate. 
Indeed, when $\calF$ is $(k + k_{0})$-strongly identifiable
it can be inferred from Theorem 6.3 in~\citep{heinrich2018} that, 
\begin{align}
\label{eq:heinrich_kahn_rate}
\bbE \big[W_r(\widebar{G}_{n}, G_{0}^{n})\big] \lesssim \left(\frac{\log n}{ n}\right)^{1/2r},
\end{align}
where we fix $r = k+k_{0} - 2k_{*}+1$
throughout the remainder of this section.
Furthermore, the above rate is minimax optimal 
up to a polylogarithmic factor,
but is markedly slower than its pointwise
analogue discussed in Section~\ref{sec:pointwise_strongly_identifiable}. 
It implies that the atoms of $\widebar{G}_{n}$ with nonvanishing weights tend to those of $G_{0}^{n}$ at this same slow rate. 
In contrast, we will show that the uniform convergence
rates of individual components of the MLE can
be sharpened. 
Similarly to the previous subsection, however, our
results will rely on the additional condition that the mixing proportions 
of $G_0^n, G_*$ are uniformly bounded below by a small constant $c_0 > 0$.
While this condition was not needed in the work of~\citet{heinrich2018}, 
we require it for our proof technique. As a result, 
we focus on deriving convergence rates for the penalized
MLE~$\widehat G_n$. 

Given $k'\in[k]$, let $G = \sum_{i = 1}^{k'} p_{i} \delta_{\theta_{i}} \in \calE_{k'}(\Theta)$ 
and $G' = \sum_{i = 1}^{k_0} p_{i}' \delta_{\theta_{i}'}\in \calE_{k_0}(\Theta)$.
We again partition the supports of these measures
into Voronoi cells, which are now generated
by the atoms of the measure $G_*$ rather than $G_0^n$:  
$$\mathcal{A}_{j}(G) = 
  \big\{i \in [k']: |\theta_{i} - \theta_{j}^{*}| \leq |\theta_{i} - \theta_{\ell}^{*}| \ \forall \ell \neq j
  \big\},$$
for all $j \in [k_{*}]$. 
With this notation in place, we define the following 
loss function over $\calO_k(\Theta)$, 
\begin{align}
    & \widetilde{W}(G, G')   = \inf_{\boldsymbol{q} \in \Pi(G, G')} \left\{\sum_{l = 1}^{k_{*}} \sum_{(i, j) \in \mathcal{A}_{l}(G) \times \mathcal{A}_{l}(G')} q_{ij} |\theta_{i} - \theta_{j}'|^{|\mathcal{A}_{l}(G)| + |\mathcal{A}_{l}(G')|  - 1}  + \sum_{(i, j) \not \in \cup_{l = 1}^{k_{*}} \mathcal{A}_{l}(G) \times \mathcal{A}_{l}(G')} q_{ij} \right\}. \label{eq:uniform_metric}
\end{align}
$\widetilde W$ may be viewed as a generalized optimal transport cost, 
whose ground cost depends on the measures $G,G'$ via
the exponent $|\mathcal{A}_{l}(G)| + |\mathcal{A}_{l}(G')|  - 1$.
In the special case where $k_*=1$, 
this exponent is given by $k+k_0-1$, and $\widetilde W$
is then equal to $W_r^r$.
On the other hand, when $k_* > 1$, 
it can be seen similarly as in previous subsections that,
\begin{align}
\label{eq:si_uniform_metric_inequality}
\widetilde{W} 
    \gtrsim W_r^r, ~~ \text{and } ~  \sup_{\substack{G\neq G'}}
\frac{\widetilde{W}(G, G')}{W_r^r(G, G')} = \infty.
\end{align}
Therefore, the loss function $\widetilde{W}$ is stronger than the Wasserstein distances used by~\citet{heinrich2018}. 
The main result of this section is the following
convergence rate under $\widetilde W$. 
\begin{theorem}
\label{theorem:uniform_rates_strong_identifiable}
Let $k \geq k_0 \geq k_*$ and $c_0 > 0$.
Assume 
that $G_* \in \calE_{k_*,c_0}(\Theta)$
and $G_0^n \in \calE_{k_0,c_0}(\Theta)$
for all $n \geq 1$. Furthermore, assume that $\calF$ is $(k + k_{0})$-strongly identifiable,
and satisfies conditions~\hyperref[assm:holder]{\textbf{A($k+k_0$)}} 
and~\ref{assm:bracket}. Then, there exist constants
$C,\epsilon > 0$, depending only on $\calF,k,c_0$, 
such that for all $n\geq 1$ satisfying $\widetilde W(G_0^n, G_*) 
\leq \epsilon$, we have
\begin{align*}
  \bbE \Big[\widetilde{W}(\widehat{G}_{n}, G_{0}^{n})\Big] \leq C \frac{\log n}{\sqrt n}. 
\end{align*}
\end{theorem}
In view of equation~\eqref{eq:si_uniform_metric_inequality}
and the existing minimax lower
bound of~\citet{heinrich2018} under the Wasserstein distance, 
it can immediately be deduced that the convergence
rate in Theorem~\ref{theorem:uniform_rates_strong_identifiable}
is minimax optimal, up to a logarithmic factor. 

The proof of Theorem~\ref{theorem:uniform_rates_strong_identifiable} 
appears in Appendix~\ref{subsec:proof:theorem:uniform_rates_strong_identifiable}. Our main technical contribution is Lemma~\ref{lemma:uniform_bound}
therein,
which provides an upper bound on $\widetilde W(G, G_0)$ in terms
of the Kolmogorov-Smirnov distance between the distributions of $p_G$
and $p_{G_0}$. Similarly to~\cite{heinrich2018}, 
we derive our upper bound by placing the atoms of $G$ and $G_0$
into an ultrametric tree, and using it
to construct a nearly optimal coupling $\bq$ in the definition of $\widetilde W$. These derivations are facilitated by the assumption $\Theta \subseteq \bbR$, but we expect that similar conclusions also hold for  
strongly identifiable families with multidimensional parameter spaces. 

Theorem~\ref{theorem:uniform_rates_strong_identifiable}
may be interpreted similarly
as in previous sections, thus we only provide
an example. In the sequel, 
we ignore polylogarithmic factors.
For all $l \in [k_*]$,  notice that
\begin{equation} 
\label{eq:uniform_cardinalty_ineq_example}
|\mathcal{A}_{l}(\widehat{G}_{n})| \leq k - k_{*} + 1,\quad |A_{l}(G_{0}^{n})| \leq k_{0} - k_{*} + 1.
\end{equation}
When these inequalities are both  achieved
by the same index $\bar l \in [k_*]$, we find that
for every 
 $i \in \calA_{\bar l}(\hat G_n)$, 
there exists $j \in \calA_{\bar l}(G_0^n)$ such that,
up to taking subsequences, 
the rate of~\citet{heinrich2018} is achieved:
$$|\htheta_i^n - \theta_j^0| \lesssim n^{-\frac{1}{2r}}.$$
However, the remaining $k_*-1$
atoms of the penalized MLE converge
uniformly at the parametric rate $n^{-1/2}$,
which could not have been deduced from equation~\eqref{eq:heinrich_kahn_rate}. 
Furthermore, we emphasize that this setting---in which
all redundant atoms of $\widehat G_n$ and $G_0^n$
are concentrated near a single atom of $G_*$---is
the only case where a subset
of the atoms of $\widehat G_n$
achieve the worst-case rate
predicted by~\citet{heinrich2018}. 
Indeed, when the inequalities~\eqref{eq:uniform_cardinalty_ineq_example}
are strict, 
rates faster than $n^{-1/2r}$ are achieved
by all atoms of~$\widehat G_n$. 
 
\section{Discussion}
The aim of our work has been to sharpen known convergence
rates of the MLE for estimating individual parameters of a finite mixture model. Our key observation was that the Wasserstein distance, despite being an elegant tool for metrizing the space of mixing measures,
is not well-suited to capturing the heterogeneous convergence behaviour
of individual mixture parameters. We
instead proposed new loss functions which achieve this goal. 
Our theoretical results are supported by a simulation study, 
which is deferred to Appendix~\ref{sec:simulations}. 

Our analysis has focused on maximum likelihood-based
estimators, whose computation involves the nonconvex
optimization problem~\eqref{eq:mle}. Despite significant
recent advances in the theoretical understanding
of the EM algorithm for approximating the MLE in finite mixtures  \citep{balakrishnan2017,dwivedi2020, kwon2019global, Raaz_Ho_Koulik_2018_second}, 
we make no claims that such approximations obey the asymptotics
described in this paper, leaving open a potential gap
between theory and practice. 
The method of moments provides
a practical alternative to the MLE, 
which is minimax optimal for certain classes of finite mixture models under the Wasserstein distance~\citep{wu2020a,doss2020}. 
We leave open the question of characterizing the
risk of moment-based estimators under the loss
functions proposed in our work. 

In Section~\ref{sec:uniform_settings}, 
we obtained uniform convergence rates
for strongly identifiable
mixtures with mixing proportions bounded
away from zero. We leave open the question 
of determining whether this constraint can be removed. 

Finally, we derived both pointwise and uniform
convergence rates for strongly identifiable
mixtures, however we restricted our analysis
of location-scale Gaussian mixtures to the pointwise case.
Obtaining 
uniform convergence rates 
for such models
remains an important open problem,
which has not been studied 
beyond the special case of two component models~\citep{hardt2015,manole2020a}.
While this setting is beyond the scope of our work, 
we expect that considerations about the heterogeneity
of parameter estimation, similar
to those studied in this paper, would arise in such models as well. 

\subsection*{Acknowledgements}
TM was supported
in part by the Natural Sciences and Engineering Research Council of Canada, through a PGS D
scholarship. NH gratefully
acknowledges research gifts by UT Austin ML grant.

\bibliographystyle{apalike}
\bibliography{main}


\begin{appendix}

\begin{center}
{\bf \large Supplement to ``Refined Convergence Rates for Maximum Likelihood Estimation under Finite Mixture Models''}
\end{center}

In this supplementary material, we 
provide all proofs of results stated in the main 
text (Appendix~\ref{sec:proof}). 
We also state and prove certain results 
which were deferred from the main text (Appendix~\ref{sec:auxiliary}), and
provide a simulation study to illustrate
the various convergence rates that were derived
in this paper (Appendix~\ref{sec:simulations}). 
 
\section{Proofs}
\label{sec:proof}

\subsection{Proof of Theorem~\ref{thm:density_estimation_rate_penalized}}
 
Theorem~\ref{thm:density_estimation_rate_penalized}(i)
is an immediate consequence of Theorem 7.4 of~\citet{vandegeer2000}, which provides a generic 
exponential inequality for the Hellinger loss of
nonparametric maximum likelihood density estimators, 
under mere conditions on the bracketing integral
$\calJ_B(\epsilon, \widebar\calP_k^{1/2}(\Theta,\epsilon), \nu)$. 
The application of this result to finite mixture models has
previously
been discussed by~\citet{ho2016EJS,  ho2016convergence}.

Theorem~\ref{thm:density_estimation_rate_penalized}(ii)
also follows by the same proof technique
as Theorem 7.4 of~\citet{vandegeer2000}, 
with modifications to account
for the presence of the penalty
in the definition of $\hat G_n$. 
An analogue of this result
was previously proven by~\citet{manole2021b}, 
though with different conditions on 
the tuning parameter $\xi_n$. 
For completeness, we provide a self-contained
proof of~Theorem~\ref{thm:density_estimation_rate_penalized}(ii), under the conditions on $\xi_n$ required
for our development.

As in~\citet{vandegeer2000}, we shall
reduce the problem to controlling the increments
of the empirical process
$$\nu_n(G) = \sqrt n \int_{\{p_{G_0}>0\}}
\frac 1 2 \log \frac {\bar p_G}{p_{G_0}} ~
d(P_n-P_{G_0}),$$
where we recall that $\bar p_G = (p_G+p_{G_0})/2$, and
we denote by $P_G = \int p_G d\nu$
the distribution induced by $p_G$, for any $G \in \calO_k(\Theta)$. 
Furthermore, $P_n = (1/n)\sum_{i=1}^n \delta_{X_i}$
denotes the empirical measure.
Our main technical tool will be the following
special case of Theorem 5.11 (\citet{vandegeer2000}; see also Lemma 7.2--7.3 therein). 
\begin{theorem}[Theorem 5.11 \citep{vandegeer2000}]
\label{thm:van_de_geer}
Let $R > 0$ and $k \geq 1$. Given 
$\calG \subseteq \calO_{k}(\Theta)$,
let $G_0 \in \calG$.
Furthermore, given a universal constant $C > 0$, let $a, C_1 > 0$ be chosen such that
\begin{equation}
\label{thm511_2}
a \leq C_1\sqrt n R^2 \wedge 8\sqrt n R,
\end{equation}
and,
\begin{equation}
\label{thm511_5}
a \geq\sqrt{C^2(C_1+1)} \left(\int_0^R \sqrt{H_B 
\left(\frac u {\sqrt 2}, \Big\{p_G:G \in \calG, h(\bar p_G, p_0) \leq R\Big\}, \nu \right)} du \vee  R\right) , 
\end{equation}
Then, 
\[
\bbP\left\{\sup_{\substack{G \in \calG \\ h(\bar p_G, p_{G_0}) \leq R}} |\nu_n(G)| \geq a\right\}
 \leq C \exp\left(-\frac{a^2}{C^2(C_1+1)R^2}\right).
 \] 
\end{theorem}

We are now in a position to prove the claim.

{\bf Proof of Theorem~\ref{thm:density_estimation_rate_penalized}(ii).} 
Let $G_0 \in \calO_{k,c_0}(\Theta)$.
By a straightforward
modification of Lemma~4.1 of~\citet{vandegeer2000}, we have
 \begin{equation} 
 \label{eq:basic}
h^2\left(\bar p_{\hat G_n}, p_{G_0}\right) 
 \leq \frac 1 {\sqrt n}\nu_n(\hat G_n) + \frac{\xi_n \pen(G_0)}{4n}.
\end{equation}  
Let $u > \gamma_n = L \log n/\sqrt n$, where
$L$ is the constant in assumption~\ref{assm:bracket}. In view of 
equation~\eqref{eq:basic}, and the fact
that $h^2(p_G, p_{G_0}) \leq 4 h(\bar p_G, p_{G_0})$
for all $G \in \calO_k(\Theta)$ (cf. Lemma 4.2 of~\citet{vandegeer2000}), we have
\begin{align*}
\bbP\left\{h(p_{\hat G_n}, p_{G_0}) > u\right\} 
 \leq \bbP\left\{h(\bar p_{\hat G_n}, p_{G_0}) > u/4\right\}
& \nonumber \\
& \hspace{-5 em} \leq 
\bbP\left\{
\sup_{\substack{G \in \calO_{k,c_0}(\Theta) \\ h(\bar p_G, p_0) > u/4}} n^{-\frac 1 2} \nu_n(G) + \frac {\xi_n\pen(G_0)} {4n} - h^2(\bar p_G, p_{G_0})\geq 0 \right\}. 
\end{align*} 
Let $\calS = \min\{s: 2^{s+1}u/4 > 1\}$. Then,
\begin{align*}
\bbP\Bigg\{\sup_{\substack{G \in \calO_{k,c_0}(\Theta) \\ h(\bar p_G, p_{G_0}) > u/4}} & n^{-\frac 1 2} \nu_n(G) + \frac {\xi_n\pen(G_0)} {4n}  - h^2(\bar p_G, p_0)\geq 0 \Bigg\}  \\
 &\leq \sum_{s=0}^{\calS} \bbP\left\{\sup_{\substack{G \in \calO_{k,c_0}(\Theta) \\ h(\bar p_G, p_{G_0}) \leq (2^{s+1})u/4}} \nu_n(G)  \geq \sqrt n 2^{2s}\left(\frac{u}{4}\right)^2 - \frac {\xi_n\pen(G_0)} {4\sqrt n}  \right\}.
\end{align*}
We have thus reduced the problem
to that of bounding the supremum of the empirical
process $\nu_n$, for which we
shall invoke Theorem~\ref{thm:van_de_geer}. Let $R=2^{s+1}u$, $C_1=15$,  and
$$a = \sqrt n 2^{2s}\left(\frac{u}{4}\right)^2 -\frac {\xi_n\pen(G_0)} {4\sqrt n}.$$
It can be directly verified that 
condition~\eqref{thm511_2} holds for
all $s=0, \dots, \calS$. 
To further 
show that condition \eqref{thm511_5} holds, note that
\begin{align*}
\int_0^{2^{s+1}u} \sqrt{H_B\left(\frac t {\sqrt 2}, \widebar\calP_k^{ 1/2}\left(\Theta,2^{s+1}\frac{t}{4}\right), \nu\right)} dt \vee 2^{s+1}u & \\
 & \hspace{- 4 em} \leq \sqrt 2 \int_0^{2^{s+\frac 1 2}u} \sqrt{H_B\left(t, \widebar\calP_k^{1/ 2}\left(\Theta, 2^{s+\frac 1 2}t\right), \nu\right)} dt \vee 2^{s+1}u\\
 & \hspace{- 4 em} \leq 2 \calJ_B\left(2^{s+1}u, \widebar\calP_k^{1/2}(\Theta,2^{s+1}u), \nu\right)
 \leq 2 J \sqrt n 2^{2s+ 1} u^2,
\end{align*} 
where we invoked condition~\ref{assm:bracket}.
Now, notice that $\rho(G_0)$ is bounded
above by a universal constant $L_0 > 0$ depending only
on $k,c_0$, irrespective of the choice of 
$G_0 \in \calO_{k,c_0}(\Theta)$. Furthermore, we have
$\sqrt n \gamma_n^2 \asymp (\log n)^2/\sqrt n$,  
and $\xi_n/\sqrt n \asymp \log n/\sqrt n$, thus
for all $u > \gamma_n$, 
the second term in the definition of $a$ is of lower
order than the first. Deduce that there exists 
a constant $N > 0$, depending only on $L_0,c_1,k$
such that for all $n \geq N$, 
$$a \geq \frac 1 2 \sqrt n 2^{2s}(u/4)^2 = 
\sqrt n 2^{2s-5} u^2 
\geq \sqrt{C_0} \cdot \big( 2J\sqrt n 2^{2s+1}u^2\big),$$
for a sufficiently small 
choice of the universal constant $J > 0$.
We may therefore invoke Theorem~\ref{thm:van_de_geer}, to deduce
that for all $n \geq N$, 
\begin{align*}
\bbP\left\{h(p_{\hat G_n}, p_{G_0}) > u\right\} 
 &\leq   \sum_{s=0}^{\calS} \bbP\left\{\sup_{\substack{\calO_{k,c_0}(\Theta) \\ h(\bar p_G, p_{G_0}) \leq (2^{s+1})u/4}} \nu_n(G)  \geq \sqrt n 2^{2s-5}u^2 \right\}\\
 & \leq C\sum_{s=0}^{\infty} \exp\left\{-\frac 1 {16C^2 2^{2s+2}\gamma_n^2}\left[\sqrt n 2^{2s-5}u^2\right]^2\right\} \\
 & \leq C\sum_{s=0}^{\infty}\exp\left\{\frac {n 2^{2s-16} u^2} {C^2}  \right\} \\  
 & \leq c \exp(-nu^2/c),
\end{align*} 
for a large enough constant $c > 0$. 
It follows that, for all $n \geq N$, 
\begin{align*}
\bbE h(p_{\hat G_n}, p_{G_0})
 = \int_0^\infty \bbP(h(p_{\hat G_n}, p_{G_0}) \geq u) du 
 \leq \gamma_n + 
   c \int_{\gamma_n}^\infty 
  \exp\left\{- \frac{nu^2}{c}\right\} du 
 \leq  
   c' \gamma_n,
  \end{align*}
for another universal constant $c' > 0$.
Since the Hellinger distance is bounded above by 1, it is clear that the above display holds for all $n \geq 1$, up to modifying the constant $c'$ in terms of $N$. 
Furthermore, the above calculation is clearly uniform in the $G_0$ under consideration, so the claim follows.
\qed 

\subsection{Proof of Proposition~\ref{prop:bded_mixing}} 
We shall require a bound on the
log-likelihood ratio statistic based
on the MLE $\widebar G_n$. 
Such a bound is implicit
in the proof of~Theorem~7.4 of~\citet{vandegeer2000}. 
Specifically, the following can be deduced
from their Corollary~7.5.
\begin{proposition}[Corollary 7.5~\citet{vandegeer2000}]
\label{prop:lrt}
Assume that condition~\ref{assm:bracket}
holds. Then, given  $k \geq 1$, there exists a constant $C > 0$
depending on $k,d$ and $\calF$, 
such that for all $u \geq L (\log n / n)^{1/2}$, 
$$\sup_{G_0 \in \calO_k(\Theta)}
\bbP_{G_0}\left(\int \log\frac{p_{\widebar G_n}}{p_{G_0}}dP_n
\geq u^2\right) 
\leq C \exp\left(-\frac {n u^2}{C^2}\right).$$
\end{proposition}
Let $G_0 \in \calO_{k,c_0}(\Theta)$. After possibly replacing
$C$ by $C \vee L$, apply Proposition~\ref{prop:lrt}
with $u = C \sqrt{\log n / n}$ to deduce that 
$$\ell_n(\widebar G_n) - \ell_n(G_0) \leq C^2\log n,$$
with probability at least $1-C/n$.
Now, by definition of the penalized MLE
$\hat G_n$ and of the non-penalized MLE $\widebar G_n$, 
we have
\begin{align*}
0\leq \Big[ \ell_n(\hat G_n) 
  - \ell_n(G_0)\Big] 
   + \xi_n \left[  \rho(\hat G_n) - \rho(G_0)\right]
 &\leq \Big[\ell_n(\widebar G_n) - \ell_n(G_0)\Big]
  +  \xi_n \left[  \rho(\hat G_n) - \rho(G_0)\right]\\
 &\leq C^2\log n
  +  \xi_n \left[  \rho(\hat G_n) - \rho(G_0)\right],
\end{align*}
with probability at least $1-C/n$. 
Therefore, since $\xi_n \geq \log n$, we obtain
$$\rho(\hat G_n) \geq -C^2 + \rho(G_0)
\geq -C^2 + k_0\log c_0 = - C_1,$$
where $C_1 = C^2 + k_0\log(1/c_0) > 0$. 
By definition of $\pen$, it must follow that
$$\hat p_i^n \geq \exp(-C_1),\quad i=1, \dots, \hat k_n,$$
with probability at least $1-C/n$. The claim follows
with $c = \exp(C_1)\vee C$. \qed 

\subsection{Proof of Theorem~\ref{theorem:rate_MLE_strong_identifiability}}
\label{subsec:proof:theorem:rate_MLE_strong_identifiability}
The claim will follow from the following result, 
relating the discrepancy $\calD(G,G_0)$ to the Total Variation
distance between the corresponding
densities $p_G$ and $p_{G_0}$. 
\begin{lemma}
\label{lemma:key_bound_strong_identifiability}
Assume 
the same conditions as
Theorem~\ref{theorem:rate_MLE_strong_identifiability}.
Then, there exists a constant $c > 0$ depending
on $G_0, d, k, \calF$, such that for any $G \in \calO_k(\Theta)$, 
\begin{align}
    V(p_{G}, p_{G_{0}}) \geq c \calD(G,G_{0}). \label{key_inequality_strong_identifiability}
\end{align} 
\end{lemma} 
Recall that we have assumed condition~\ref{assm:bracket}. Therefore, by combining Lemma~\ref{lemma:key_bound_strong_identifiability}
with Theorem~\ref{thm:density_estimation_rate_penalized}(i)
and the well-known inequality $V \leq h$, 
we deduce that
$$\bbE \calD(\widebar G_n, G_0) \lesssim \bbE V(p_{\widebar G_n}, p_{G_0})
\leq  \bbE h(p_{\widebar G_n}, p_{G_0})
\lesssim \sqrt{\frac{\log n}{n}},$$ 
as claimed. It thus remains to prove 
Lemma~\ref{lemma:key_bound_strong_identifiability}. 

\paragraph{Proof of Lemma~\ref{lemma:key_bound_strong_identifiability}.}
Our proof proceeds using a similar
argument as that of~\citet{ho2016EJS}, 
though with key differences to account for 
our choice of loss function.
We will prove that 
\begin{equation} 
\label{eq:pf_si_local_limit}
\lim_{\delta \to 0}
\inf_{\substack{G \in \calO_k(\Theta) \\ 
\calD(G, G_0) \leq \delta}} 
\frac{V(p_{G}, p_{G_0})}{\calD(G, G_0)}
 > 0.
 \end{equation}
This implies a local version of the claim, namely that there exist constants
$\delta_0,C > 0$ such that
for all $G \in \calO_{k}(\Theta)$ satisfying $\calD(G,G_0) \leq \delta$,
\begin{equation}
    \label{eq:pf_si_local}
    \calD(G, G_0) 
\leq C V(p_G, p_{G_0}).
\end{equation} 
We begin by showing how this local inequality
leads to the claim, and we will then prove equation \eqref{eq:pf_si_local_limit}. 
Taking equation \eqref{eq:pf_si_local_limit} for granted, it suffices to prove
\begin{equation}
\label{eq:si_global}
    \inf_{\substack{G \in \calO_{k}(\Theta) \\ \calD(G, G_0) \geq \delta_0}}
    \frac{V(p_G, p_{G_0})}{\calD(G, G_0)} > 0.
\end{equation}
Suppose by way of a contradiction that the above display does not
hold. Then, there exists a sequence of mixing measures
$G_n \in \calO_{k}(\Theta)$ with $\calD(G_n, G_0) \geq \delta_0$
such that $
    \frac{V(p_{G_n}, p_{G_0})}{\calD(G_n, G_0)} \to 0.
$
Since the parametric family $\calF$ is assumed to be 2-strongly
identifiable, the model $\{p_G: G \in \calO_{k}(\Theta)\}$ is identifiable,
thus the map
$$(G, G') \in \calO_k(\Theta) \times \calO_k(\Theta) \mapsto V(p_G, p_{G'})$$
defines a metric on $\calO_k(\Theta)$. 
Since this metric is bounded, 
the sequence $\{G_n\}$ admits a  subsequence 
converging
to some mixing measure $\widebar G \in \calO_{k}(\Theta)$. For ease of exposition,
we replace this subsequence by the entire sequence $G_n$ in what follows,
thus we have $V(p_{G_n}, p_{\widebar G}) \to 0$.
Now, notice that $\calD(\widebar G, G_0) \geq \delta_0$ 
by definition of $G_n$. Furthermore, $V(p_{G_n}, p_{G_0}) \to 0$
by assumption.
Combining these facts leads to
$V(p_{G_0}, p_{\widebar G}) = 0$, and hence
$\widebar G = G_0$, which contradicts the fact that
$\calD(\widebar G, G_0) > 0$,
and hence proves equation \eqref{eq:si_global}.

It remains to prove the local inequality \eqref{eq:pf_si_local_limit}. 
We again assume by way of a contradiction that
there exists a sequence of mixing measures $G_n = \sum_{i=1}^{k_{n}} p_i^n \delta_{\theta_i^n} \in \calO_{k}(\Theta)$
such that $\calD(G_n, G_0) \to 0$  but
\begin{equation}
\label{eq:pf_si_contradiction}
\frac{V(p_{G_n}, p_{G_0})}{\calD(G_n, G_0)} \to 0, \quad n\to\infty.
\end{equation}
Define
$$\mathcal{A}_{j}^{n} = \mathcal{A}_{j}(G_{n}) = 
  \{i \in \{1, \ldots, k_n\}: \|\theta_{i}^{n} - \theta_{j}^{0}\| \leq \|\theta_{i}^{n} - \theta_{\ell}^{0}\| \ \forall \ell \neq j
  \}, \quad j = 1, \dots, k_0.$$
Since $k_{n} \leq k$ for all $n$, there exists a subsequence of $G_{n}$ such that $k_{n}$ does not change with $n$.
Therefore, up to replacing $G_n$ by this subsequence, we may
assume that $k_{n} = k' \leq k$ for all $n$.
Similarly, since there are only a finite number of
distinct sets $\mathcal{A}_{1}^{n} \times  \ldots \times \mathcal{A}_{k_{0}}^{n}$
over the range of $n\geq 1$, we may assume without loss of generality that 
$\mathcal{A}_{j} = \mathcal{A}_{j}^{n}$ 
does not change with $n$, for all $j=1, \dots, k_0$. 
Now, consider the decomposition
\begin{align*}
p_{G_n}(x) - p_{G_0}(x)
 &= \sum_{j: |\mathcal{A}_{j}| > 1} \sum_{i \in \mathcal{A}_{j}} p_i \Big( f(x|\theta_i^n) - f(x|\theta_j^0)\Big)
 \\ &+ \sum_{j: |\mathcal{A}_{j}| = 1} \sum_{i \in \mathcal{A}_{j}} p_i \Big( f(x|\theta_i^n) - f(x|\theta_j^0)\Big)
+ \sum_{j = 1}^{k_{0}} (\bar p_j^n - p_j^0) f(x|\theta_j^0)\\
 & : = A_{n,1}(x) + A_{n,2}(x)+ B_n(x),
\end{align*} 
where we write $\bar p_j^n = \sum_{i\in \calA_j} p_i^n$ for all
$j \in [k_0]$. 
By a Taylor expansion to second order, notice that
\begin{align*}
A_{n,1}(x) 
 = \sum_{j: |\mathcal{A}_j| > 1} \sum_{i \in \mathcal{A}_j} p_i\left[ (\theta_i^n - \theta_j^0)^\top \frac{\partial f}{\partial \theta}(x|\theta_j^0) + 
 \frac 1 2 (\theta_i^n - \theta_j^0)^\top \frac{\partial^2 f}{\partial \theta^2}(x|\theta_j^0)  (\theta_i^n - \theta_j^0)\right]  + R_{n,1}(x)
\end{align*}
where $R_{n,1}(x)$ is a Taylor remainder satisfying
\begin{equation} 
\label{eq:pf_si_pointwise_first_remainder}
\|R_{n,1}\|_{L^\infty(\nu)} \lesssim \sum_{j:|\calA_j| > 1} \sum_{i\in \calA_j} p_i^n \norm{\theta_i^n - \theta_j^0}^{2+\gamma},
\end{equation} 
for some $\gamma > 0$, due
to condition~\hyperref[assm:holder]{\textbf{A($2$)}}.
Furthermore, by a Taylor expansion to first order, we also have 
\begin{align*}
    A_{n,2}(x)=\sum_{j: |\mathcal{A}_{j}| = 1} \sum_{i \in \mathcal{A}_{j}} p_i (\theta_i^n - \theta_j^0)^\top \frac{\partial f}{\partial \theta}(x|\theta_j^0) 
    + R_{n,2}(x),
\end{align*}
where, again, the Taylor remainder $R_{n,2}$ satisfies
\begin{equation} 
\label{eq:pf_si_pointwise_second_remainder}
\|R_{n,2}\|_{L^\infty(\nu)}\lesssim  \sum_{j:|\calA_j|= 1} \sum_{i\in \calA_j} p_i^n \norm{\theta_i^n - \theta_j^0}^{1+\gamma},
\end{equation} 
Let $D_n = \calD(G_n, G_0)$. By equations~\eqref{eq:pf_si_pointwise_first_remainder}--\eqref{eq:pf_si_pointwise_second_remainder}
and the definition of $\calD$, we deduce that 
 $\|R_{n,\ell}\|_{L^\infty(\nu)}/D_n = o(1)$ for $\ell=1,2$. 
 Therefore, we have 
 uniformly almost everywhere in $x \in \calX$ that, 
$$\left|\frac{p_{G_n}(x) - p_{G_0}(x)}{D_n}\right| 
\asymp\left| \frac{A_{n,1}(x) + A_{n,2}(x) + B_n(x)}{D_n}\right|.$$
Notice that the ratio $(A_{n,1}(x) + A_{n,2}(x) + B_n(x))/D_n$
is a linear combination of $f(x|\theta_j^0)$ and its first two partial derivatives, 
with coefficients not depending on $x$.  
We claim that at least one of these coefficients does not tend
to zero as $n \to \infty$. 
Indeed, suppose by way of a contradiction that this is not the case. 
Then, in particular the coefficients corresponding to the second derivatives in $A_{n,1}/D_n$ and the coefficients
corresponding to the first derivatives in $A_{n,2}/D_n$ must vanish, and the  absolute sum of any subset of these coefficients
must vanish, implying the following display,
$$\frac 1 {D_n} \left[ 
 \sum_{j: |\mathcal{A}_j| > 1} \sum_{i \in \mathcal{A}_j} p_i
 \norm{\theta_i^n-\theta_j^0}^2 
 + \sum_{j: |\mathcal{A}_j| = 1} \sum_{i \in \mathcal{A}_{j}} p_i
 \norm{\theta_i^n-\theta_j^0}
\right] \longrightarrow 0.$$ 
The definition of $D_n$ then implies that
$$\frac{\sum_{j=1}^{k_0} |\bar p_j - p_j^0|}{D_n} \longrightarrow 1.$$
We deduce that at least one coefficient in the linear combination
$B_n(x)/D_n$ does not tend to zero, which is a contradiction. 
Thus, there indeed exists at least one coefficient
in the linear combinations $A_{n,\ell}(x)/D_n$, $B_n(x)/D_n$, $\ell=1,2,$
which does not vanish. Let $m_n$ denote the greatest absolute value of these nonzero
coefficients, and set $d_n = 1/m_n$. Then, there must exist scalars
  $\alpha_i \in \bbR$ and vectors $\beta_{j},\nu_{j} \in \bbR^d$,
  $j=1,\dots,k_0$, not all of which are zero,
  such that for almost all $x \in \calX$,
\begin{align}
\label{eq:pf_si_coefs}
\begin{split}
\frac{d_n A_{n,1}(x)}{D_n} + \frac{d_n A_{n,2}(x)}{D_n}
 &\longrightarrow \sum_{j=1}^{k_0}
 \left[\beta_{j}^\top \frac{\partial f}{\partial\theta}(x|\theta_j^0) 
  + \nu_{j}^\top \frac{\partial^2 f}{\partial\theta^2}(x|\theta_j^0) \nu_{j}   \right]  \\
\frac{d_n B_n(x)}{D_n} &\longrightarrow \sum_{j=1}^{k_0} \alpha_j f(x|\theta_j^0).
\end{split}
  \end{align}
On the other hand, the assumption \eqref{eq:pf_si_contradiction} and the fact that $d_n$ are uniformly bounded implies that
$$
d_n\frac{V(p_{G_n}, p_{G_0})}{D_n} = \int d_n \left|\frac{A_{n,1}(x) + A_{n,2}(x) + B_n(x)}{D_n}\right|dx \to 0.$$ 
By Fatou's Lemma combined with equation \eqref{eq:pf_si_coefs}, it follows that for almost all $x \in \calX$,
$$ \sum_{j=1}^{k_0} \left[\alpha_j f(x|\theta_j^0) + 
\beta_{j}^\top \frac{\partial f}{\partial\theta}(x|\theta_j^0) 
  + \nu_{j}^\top \frac{\partial^2 f}{\partial\theta^2}(x|\theta_j^0) \nu_{j}   \right]  =0.$$
Since the coefficients $\alpha_j,\beta_{j}, \nu_{j}$ are not all zero, the above display contradicts
the second-order strong identifiability assumption on the parametric family $\calF$. It follows that equation \eqref{eq:pf_si_contradiction} could not have held,
whence the claim \eqref{eq:pf_si_local_limit} is proved. This completes the proof. \qed 
\subsection{Proof of Theorem~\ref{theorem:rate_MLE_Gaussian}}
\label{subsec:proof:theorem:rate_MLE_Gaussian}
We will prove Theorem~\ref{theorem:rate_MLE_Gaussian}
as a consequence of the following
upper bound of $\widebar {\calD}$ by the Total Variation distance. 
\begin{lemma}
\label{lemma:key_bound_Gaussian}
Assume the same conditions as Theorem~\ref{theorem:rate_MLE_Gaussian},
and let $c_{0} \in (0, \min_{1 \leq j \leq k_{0}} p_j^{0})$. Then, 
there exists $C > 0$, depending
on $G_0, c_0,d,k,\Theta$ such that 
for all $G \in \mathcal{O}_{k, c_{0}}(\Theta)$,
\begin{align} 
    V(p_{G}, p_{G_{0}}) \geq C \widebar{\mathcal{D}}(G,G_{0}). \label{key_inequality_Gaussian}
\end{align}
\end{lemma}
Before proving Lemma~\ref{lemma:key_bound_Gaussian}, 
we show how it leads to the claim.
Under the conditions of
Theorem~\ref{theorem:rate_MLE_Gaussian} regarding
the parameter space $\Theta$,
it follows from Lemma 2.1 of~\citet{ho2016EJS}
(see also~\citet{ghosal2001}) 
that the location-scale
Gaussian density family $\calF$ satisfies 
$$H_B(\epsilon, \widebar \calP_k^{1/2}(\Theta,\epsilon),\nu) 
\leq C_1 \log(1/\epsilon),\quad \epsilon > 0,$$
for a constant $C_1 > 0$ depending on $d,k,\Theta$.
Given $L > 0$, it follows that for all 
$\epsilon \geq L(\log n/n)^{1/2}$,
\begin{align*} 
\calJ_B(\epsilon, \widebar\calP_k^{1/2}(\Theta,\epsilon),\nu) 
 \leq C_1 \epsilon \sqrt{\log(1/\epsilon)} 
 = C_1 \sqrt n \left(\frac \epsilon {\sqrt n}\right) \sqrt{\log(1/\epsilon)}  
 &\leq \frac{C_1 \sqrt n \epsilon^2}{L}.
\end{align*}
Condition~\ref{assm:bracket} is then
satisfied by choosing $L = C_1/J$, thus
we may apply Theorem~\ref{thm:density_estimation_rate_penalized}
and Proposition~\ref{prop:bded_mixing} in what follows. 

By Proposition~\ref{prop:bded_mixing},
there is an event $A_n$ and a constant $c > 1$ such that 
$\bbP(A_n^\cp) \leq c/n$ and 
$\hat p_i^n \geq 1/c$ for all $i=1,\dots,\hat k_n$. 
In particular, letting $c_0 = \min\{p_j^0:j\in[k_0]\} \wedge c^{-1}$, we have
$\widehat G_n \in \calO_{k,c_0}(\Theta)$ over the event $A_n$. 
Therefore, by   Lemma~\ref{lemma:uniform_bound}
and the fact that $\widebar \calD $ is 
bounded by a constant depending only
on $\diam(\Theta), k$
we arrive at
\begin{align*}
\bbE\Big[ \widebar \calD (\hat G_n, G_0^n)\Big]
 &= \bbE\left[ \widebar \calD (\widehat G_n, G_0^n) I_{A_n}\right] + 
    \bbE\left[\widebar \calD (\widehat G_n, G_0^n) I_{A_n^\cp}\right]   \\
 &\lesssim  \bbE\left[ h(p_{\widehat G_n}, p_{G_0^n}) I_{A_n}\right] + \bbP(A_n^\cp)
 \lesssim \log n / \sqrt n + 1/n
 \lesssim \log n / \sqrt n,
\end{align*}
where we used the inequality $V \leq h$ and we invoked the Hellinger
rate of convergence of $p_{\hat G_n}$, given
in Theorem~\ref{thm:density_estimation_rate_penalized}(ii). The claim follows; it thus remains to prove
Lemma~\ref{lemma:key_bound_Gaussian}.

{\bf Proof of Lemma~\ref{lemma:key_bound_Gaussian}.} We will prove the following local version of the claim:
\begin{equation} 
\label{eq:pf_gauss_claim}
\lim_{\delta\to 0} \inf_{\substack{G \in \calO_{k,c_0}(\Theta)\\ \widebar{\calD}(G, G_0) \leq \delta  }}
\frac{V(p_G, p_{G_0})}{\widebar\calD(G, G_0)} > 0.
\end{equation}
The above local statement directly leads to the claim
by the same argument as in the beginning of the 
proof of Lemma~\ref{lemma:key_bound_strong_identifiability},
and we therefore omit it. 
Our proof follows along similar lines
as the proof of Proposition 2.2 of~\citet{ho2016convergence}, 
though with key modifications to account
for our distinct loss function. We proceed with the following steps. 

\textbf{Step 1: Setup.} To prove inequality \eqref{eq:pf_gauss_claim}, 
assume by way of a contradiction that it does not
hold. Then, there exists a sequence of mixing
measures $G_n = \sum_{i=1}^{k_n}p_i^n \delta_{(\mu_i^n,\Sigma_i^n)}$
with $p_i^n \geq c_0$ for all $i \in [k_n]$, 
such that $D_n := \widebar{\calD}(G_n, G_0) \to 0$
and $V(p_{G_n}, p_{G_0}) / D_n \to 0.$ 
Furthermore, since
$k_n \leq k$ for all $n \geq 1$, 
there exists a subsequence of $G_n$
admitting a fixed number of atoms
$k_n = k' \leq k$. 
Similarly as in the proof of Theorem~\ref{theorem:rate_MLE_strong_identifiability}, 
we replace $G_n$ by such a subsequence throughout
the sequel.

Define the Voronoi diagram
$$\calA_j^n = \left\{1\leq i \leq k': \norm{\mu_i^n - \mu_j^0}
+ \norm{\Sigma_i^n - \Sigma_j^0} 
\leq \norm{\mu_i^n - \mu_\ell^0} + 
\norm{\Sigma_i^n - \Sigma_\ell^0}, \ \forall \ell \neq j\right\},
\quad j=1, \dots, k_0.$$
By the same argument as 
in the proof of Lemma~\ref{lemma:key_bound_strong_identifiability}, we may assume, up to taking a further
subsequence of $G_n$, that
the sets $\calA_j \equiv \calA_j^n$
do not change with $n$
for all $ j=1, \dots, k_0$ and all $n \geq 1$. 
Furthermore, we note that, 
since the mixing proportions of $G_n$ are bounded below
by $c_0$, the fact that $D_n \to 0$ implies 
$$\sup_{i \in \calA_j} 
\Big[ \norm{\mu_i^n - \mu_j^0} + 
    \norm{\Sigma_i^n - \Sigma_j^0}\Big] 
    \to 0, \quad
    j=1, \dots, k_0.$$
Throughout what follows, 
we  write the coordinates 
of $\mu_j^0$ and $\Sigma_j^0$
as $\mu_j^0 = (\mu_{j,1}^0, \dots, \mu_{j,d}^0)$
and $\Sigma_j^0 = (\Sigma_{j,uv}^0)_{u,v=1}^d$,
for all $j=1, \dots, k_0$, 
and similarly for $\mu_i^n, \Sigma_i^n$,
$i=1, \dots, k'$.
We also write for simplicity
$\theta_i^n = (\mu_i^n, \Sigma_i^n)$ 
and $\theta_j^0 = (\mu_j^0,\Sigma_j^0)$ for 
all $j=1, \dots, k_0$ and $i=1, \dots, k'$.

\textbf{Step 2: Taylor Expansions.} 
Similarly to the proof of Lemma~\ref{lemma:key_bound_strong_identifiability}, consider
the following representation
\begin{align*}
p_{G_n}(x) - p_{G_0}(x)
 &= \sum_{j: |\calA_{j}| > 1} \sum_{i \in \mathcal{A}_{j}} p_i^n \Big( f(x|\theta_i^n) - f(x|\theta_j^0)\Big)
 \\ &+ \sum_{j: |\calA_{j}| = 1} \sum_{i \in \mathcal{A}_{j}} p_i^n \Big( f(x|\theta_i^n) - f(x|\theta_j^0)\Big)
+ \sum_{j = 1}^{k_{0}} (\bar p_j^n - p_j^0) f(x|\theta_j^0)\\
 & : = \bar A_{n}(x) + \bar B_n(x)+ C_n(x),
\end{align*} 
where $\bar p_j^n = \sum_{i\in \calA_j} p_i^n$ for all $j\in [k_0]$. 
By repeated Taylor expansions to order $\rbar(|\calA_j^n|)$ 
for all $j=1, \dots, k_0$,
we obtain 
\begin{align*}
\bar A_{n}(x) 
 = \sum_{j: |\calA_j| > 1} \sum_{i \in \calA_j} p_i^n
 \sum_{\alpha,\beta}\frac 1 {\alpha!\beta!} 
 (\mu_i^n - \mu_j^0)^{\alpha}
 (\Sigma_i^n - \Sigma_j^0)^{\beta} \frac{
  \partial^{|\alpha|+|\beta|} f}{\partial\mu^{\alpha}\partial\Sigma^{\beta}} (x|\theta_j^0)  + R_{n,1}(x)
  =: A_{n}(x) + R_{n,1}(x),
\end{align*}
where the third summation 
in the above display
is over all multi-indices
$\alpha \in \bbN^d$
and $\beta \in \bbN^{d\times d}$
satisfying $1\leq |\alpha|+|\beta|
 := \sum_{l=1}^d \alpha_l + \sum_{l,s=1}^d \beta_{ls} \leq \rbar(|\calA_{j}|)$. 
Above, we write
 $\alpha! = \prod_{l=1}^d \alpha_l!$
 and $\beta! = \prod_{l,s=1}^d \beta_{ls}!$.
Furthermore, $R_{n,1}$ is a Taylor remainder which satisfies
$$\|R_{n,1}\|_{L^\infty(\nu)}
\lesssim \sum_{j:|\calA_j| > 1}
\sum_{i \in \calA_j} p_i^n \Big[ \norm{\mu_i^n
 - \mu_j^0}^{\rbar(|\calA_{j}|)+\gamma}
 + \norm{\Sigma_i^n - \Sigma_j^0}^{\rbar(|\calA_{j}|)+\gamma} \Big],$$
for some constant $\gamma > 0$, as a result of the H\"older smoothness over $\Theta$, up to arbitrary order, of the location-scale Gaussian parametric
family.
Now, recall the key PDE \eqref{eq:Pde_Gaussian},
which implies
that for any multi-indices
$\alpha\in \bbN^d$ and $\beta \in \bbN^{d\times d}$, 
$$\frac{\partial^{|\alpha|+|\beta|}f}{\partial\mu^{\alpha}
\partial\Sigma^{\beta}}
 = \frac 1 {2^{|\beta|}} \frac{\partial^{|\alpha|+2|\beta|} f}{\partial \mu^{\tau_0(\alpha,\beta)}},$$
where  
we denote by $\tau_0(\alpha,\beta) \in \bbN^d$
the multi-index
with coordinates
$\alpha_v + \sum_{u=1}^d (\beta_{uv} + \beta_{vu})$, $v=1, \dots, d$. 
Notice that we may then write for all $x \in \bbR^d$,
\begin{align*}
A_{n}(x) 
 &= \sum_{j: |\calA_j| > 1} \sum_{i \in \calA_j} p_i^n
 \sum_{\substack{\alpha,\beta \\ 1 \leq |\alpha| + |\beta| \leq \rbar(|\calA_j|)}}\frac 1 {2^{|\beta|}\alpha!\beta!} 
 (\mu_i^n - \mu_j^0)^{\alpha}
 (\Sigma_i^n - \Sigma_j^0)^{\beta} \frac{
  \partial^{|\alpha|+2|\beta|} f}{\partial\mu^{\tau_0(\alpha,\beta)}} (x|\theta_j^0)  \\
  &=  \sum_{j: |\calA_j| > 1} 
  \sum_{|\tau|=1}^{2\rbar(|\calA_j|)} 
 a_{\tau,j} \frac{ \partial^{|\tau|} f}{\partial\mu^{\tau}} (x|\theta_j^0),
\end{align*} 
where for all $\tau \in \bbN^d$, we write
$$ a_{\tau,j} = \sum_{\substack{\alpha,\beta   \\ 1\leq |\alpha|+|\beta|\leq \rbar(|\calA_j|) \\ \tau_0(\alpha,\beta) = \tau}} 
 \sum_{i \in \calA_j} \frac 1 {2^{|\beta|}\alpha! \beta !}p_i^n
  (\mu_i^n - \mu_j^0)^{\alpha}
 (\Sigma_i^n - \Sigma_j^0)^{\beta}.$$
Furthermore, by a first-order
Taylor expansion in the definition of 
$\bar B_n$, we obtain
\begin{align*}
\bar B_n(x)
 &=\sum_{j: |\calA_{j}| = 1} \sum_{i \in \calA_{j}} p_i \left\{
(\mu_i^n - \mu_j^0)^\top \frac{\partial f}{\partial \mu}(x|\theta_j^0) +
\mathrm{tr}\left[\frac{\partial f}{\partial\Sigma}(x|\theta_j^0)^\top 
(\Sigma_i^n - \Sigma_j^0)\right]\right\}
+ R_{n,2}(x) \\
 & =: B_n(x) + R_{n,2}(x),
\end{align*}
where $R_{n,2}$ is a Taylor
remainder which 
satisfies,
$$\|R_{n,2}\|_{L^\infty(\nu)}\lesssim \sum_{j:|\calA_j|= 1} \sum_{i\in \calA_j} p_i^n \Big[ \norm{\mu_i^n - \mu_j^0}^{1+\gamma}
 + \norm{\Sigma_i^n - \Sigma_j^0}^{1+\gamma}\Big].$$
 Similarly as the term $A_{n}$, 
 we may explicitly
 rewrite 
 $B_n$ as a linear combination
 of the first- and second-order partial derivatives
 of the density $f$ with respect to $\mu$, 
\begin{align*} 
B_{n}(x)
 &= \sum_{j: |\calA_{j}| = 1} \sum_{i \in \calA_{j}} p_i \left\{
(\mu_i^n - \mu_j^0)^\top \frac{\partial f}{\partial \mu}(x|\theta_j^0) +
\frac 1 2 \mathrm{tr}\left[\frac{\partial f}{\partial\mu\partial\mu^\top}(x|\theta_j^0)^\top 
(\Sigma_i^n - \Sigma_j^0)\right]\right\}   \\
&  =  \sum_{j: |\calA_j|= 1} 
  \sum_{|\kappa|=1}^{2} 
 b_{\kappa,j} \frac{ \partial^{|\kappa|} f}{\partial\mu^{\kappa}} (x|\theta_j^0),
\end{align*} 
where 
$$ b_{ \kappa,j} = \sum_{\substack{\alpha,\beta   \\  |\alpha|+|\beta|= 1 \\ \tau_0(\alpha,\beta) = \kappa}} 
 \sum_{i \in \calA_j} \frac 1 {2^{|\beta|}}p_i^n
  (\mu_i^n - \mu_j^0)^{\alpha}
 (\Sigma_i^n - \Sigma_j^0)^{\beta}.$$ 
Notice that 
the conditions on the remainder
terms $R_{n,1},R_{n,2}$ together with the definition
of $D_n$ readily imply that,
uniformly in  $x \in \bbR^d$,
\begin{align}
\left|\frac{p_{G_n}(x) - p_{G_0}(x)}{D_n} \right| \asymp 
\left|\frac{A_{n}(x) + B_{n}(x) + C_n(x)}{D_n}\right|.
\end{align} 
Letting
$c_{j} = \bar p_j^n - p_j^0$, 
it can be seen that 
the right-hand side of the above display is a linear
combination of partial derivatives of $f$ with 
respect to $\mu$, with coefficients
$a_{\tau,j}/D_n, b_{\kappa,j}/D_n,
c_j/D_n$, $j=1, \dots, k_0$, where
$\tau$ and $\kappa$
vary over the aforementioned ranges.
In the next step, 
we will show that not all
of these coefficients
decay to zero. 

\textbf{Step 3: Nonvanishing Coefficients.} 
Assume by way
of a contradiction that 
all coefficients
$a_{\tau,j}/D_n, b_{\kappa,j}/D_n,
c_j/D_n$ tend to zero. 
Define the following quantities,  
 \begin{align*}
D_{n,1} &=\sum_{j: |\calA_{j}| > 1} \sum_{i \in \calA_{j}} p_{i}^n \left\{ \|\mu_{i}^n - \mu_{j}^{0}\|^{\bar{r}(|\calA_{j}|)} + 
\|(\Sigma_{i,uu}^n - \Sigma_{j,uu}^0)_{1 \leq u \leq d}\|^{\frac{\rbar(|\calA_j|)}{2}}
\right\}, \\
D_{n,2} &=\sum_{j: |\calA_{j}| > 1} \sum_{i \in \calA_{j}} p_{i}^n \| \left(\Sigma_{i,uv}^n - \Sigma_{j,uv}^0\right)_{1 \leq u \neq v \leq d}\|^{\frac{\rbar(|\calA_j|)}{2}},\\
D_{n,3} &=  \sum_{j: |\calA_{j}| = 1} \sum_{i \in \calA_{j}} p_{i}^n \left(\|\mu_{i}^n - \mu_{j}^{0}\| + \|\Sigma_{i}^n- \Sigma_{j}^{0}\| \right), \nonumber \\
D_{n,4} &=  \sum_{j = 1}^{k_{0}} \left|\bar  p_{j}^n - p_{j}^{0}\right|.  
\end{align*}
In the special case $d=1$, $D_{n,2}$ is understood to 
be identically equal to zero. 
Note that
there must exist
$1 \leq i \leq 4$ such that
$D_{n,i}/D_n \not\to 0$. 
We will consider four cases
according to which of the terms $D_{n,i}$ 
dominates $D_n$ 

\textbf{Case 3.1: $D_{n,1}/D_n \not\to 0$.}
In this case, it must 
hold that for some indices $1 \leq j \leq k_0$
and $1 \leq u \leq d$ such that 
$\tilde D_{n,1} / D_n \not\to 0$, where
$$\tilde D_{n,1} = \sum_{i \in \calA_j}
p_i^n \Big[ |\mu_{i,u}^n - \mu_{j,u}^0|^{\rbar(|\calA_{j}|)}
 +  |\Sigma_{i,uu}^n - \Sigma_{j,uu}^0|^{\rbar(|\calA_j|)/2}\Big]
$$
Fix such $j$ and assume $u=1$ without loss of generality,
throughout the rest of this Case. 
It follows by assumption that 
$a_{\tau,j}/\tilde D_{n,1} \to 0$
for all $1 \leq |\tau| \leq \rbar(|\calA_j|)$. 
In particular, 
this property holds for all $\tau$ such that $\tau_l = 0$ for $l=2, \dots, d$.
Notice that $\tau = \tau_0(\alpha,\beta)$ 
takes the latter form if and only if $\alpha_l = \beta_{1l} = \beta_{l1} = \beta_{ls}  = 0$
for all $l,s=2, \dots, d$. 
Therefore, taking the sum over 
such multi-indices leads
to the limit
\begin{equation} 
\label{eq:pf_gaussian_presystem}
\frac  1 {\tilde D_{n,1}}\sum_{i\in \calA_j} 
\sum_{\substack{\alpha_1,\beta_{11} \\ \alpha_1+2\beta_{11}=\tau_1}} p_i^n 
\frac{1}{2^{\beta_{11}}\alpha_1!\beta_{11}!}
(\mu_{i,1}^n - \mu_{j,1}^0)^{\alpha_1}
(\Sigma_{i,11}^n - \Sigma_{j,11}^0)^{\beta_{11}}
 \to 0, \quad \tau_1=1, \dots, \rbar(|\calA_j|).
 \end{equation}
Now, define
$$\widebar m_n = \max_{i \in \calA_j} p_i^n,\quad  
  \widebar M_n =
  \max\{|\mu_{i,1}^n-\mu_{j,1}^0|, 
  |\Sigma_{i,11}^n - \Sigma_{j,11}^0|^{1/2}: 
  i\in \calA_j\}.$$
For any $i \in \calA_j$,
$p_i^n/\widebar m_n$ forms a bounded
sequence of positive real numbers. Therefore, up
to replacing it by a subsequence, 
it admits a nonnegative limit which we denote
by $z_i^2 = \lim_{n\to \infty}p_i^n/\widebar m_n$. 
We similarly define
$x_i = \lim_{n\to\infty} (\mu_{i,1}^n-\mu_{j,1}^0)/\widebar M_n,$
and
$y_i = \lim_{n\to\infty}(\Sigma_{i,11}^n-\Sigma_{j,11}^0)/2\widebar M_n^2$.
We note that, since $p_i^n \geq c_0$
due to the definition of $\calO_{k,c_0}(\Theta)$, 
the real numbers $z_i$ are nonvanishing, 
and at least one is equal to 1. Similarly,
at least one of each of the $a_i$ and
$b_i$ is equal to $1$ or $-1$. 
Furthermore,  $\tilde D_{n,1}/(\widebar m_n \widebar M_n^{\tau_1})
\not\to 0$ for any $\tau_1 = 1, \dots, \rbar(|\calA_j|)$. 
We may then divide the numerator and denominator
in equation \eqref{eq:pf_gaussian_presystem} by 
$\widebar M_n^{\tau_1} \widebar m_n$
and take $n \to \infty$, to obtain the following
system of polynomial equations
\begin{align*}
\sum_{i \in \calA_j} \sum_{\alpha_1 + 2\beta_{11} = \tau_1}
\frac{z_i^2 x_i^{\alpha_1} y_i^{\beta_{11}}}
     {\alpha_1!\beta_{11}!}=0, \quad 
     \tau_1=1 ,\dots, \rbar(|\calA_j|).
\end{align*}
By definition of $\rbar(|\calA_j|)$,
this system cannot have any nontrivial solutions, 
which is a contradiction. 

\textbf{Case 3.2: $D_{n,2}/D_n \not\to 0$.}
In this case, 
there must instead exist indices $1 \leq j \leq k_0$ and 
$1 \leq u \neq v \leq d$ for which 
$\tilde D_{n,2}/D_{n,2} \not\to 0$, where
$$\tilde D_{n,2} = \sum_{i\in \calA_j} p_i^n 
|\Sigma_{i,uv}^n-\Sigma_{j,uv}^0|^{\rbar(|\calA_j|)/2}.$$
Without loss of generality, we assume $u=1$ and $v=2$,
and fix the above choice of $j$ throughout the sequel. 
Similarly to the previous case, 
we have by assumption that $a_{\tau,j}/\tilde D_{n,2} \to 0$
for all $1 \leq |\tau| \leq \rbar(|\calA_j|)$. 
It must also follow that for all such $\tau$, 
$a_{\tau,j}/\widebar  D_{n,2} \to 0$, where
$$\widebar D_{n,2} =  \sum_{i\in \calA_j} p_i^n |\Sigma_{i,12}^n - \Sigma_{j,12}^0|^2.$$
Here, we used the fact that $|\calA_j| \geq 2$, hence
$\bar  r(|\calA_j|) \geq 4$. 
In particular, this property
holds for the value  $\tau = (2, 2, 0, \dots, 0)$, 
where we again note that this choice of $\tau$ is allowable because 
 $\rbar(|\calA_j|) \geq 4 = |\tau|$.
Therefore, 
$$\frac 1 {\widebar  D_{n,2}} \sum_{\substack{\alpha,\beta  \\\tau_0(\alpha,\beta) = \tau}} 
 \sum_{i \in \calA_j} \frac 1 {2^{|\beta|}\alpha! \beta !}p_i^n
  (\mu_i^n - \mu_j^0)^{\alpha}
 (\Sigma_i^n - \Sigma_j^0)^{\beta} \to 0.$$
Since Case 3.1 does not hold, 
we have $ D_{n,1} / \widebar D_{2,n}\to 0$. 
Therefore, under the assumption of Case 3.2, 
any term in the above
summation with $\alpha_l > 0$ or $\beta_{ll} > 0$ ($l=1,2$)
vanishes, and the preceding display
thus reduces to 
$$\frac 1 {\widebar  D_{n,2}}  
 \sum_{i \in \calA_j} p_i^n
 (\Sigma_{i,12}^n - \Sigma_{j,12}^0)^2 \to 0.$$
 By definition of $\widebar D_{n,2}$, this is a contradiction, 
 thus Case 3.2 could not have held.
 
\textbf{Case 3.3: $D_{n,3}/D_n \not\to 0$.}
By assumption, 
the coefficients 
$b_{\kappa,j}/D_{n}$ vanish
for all multi-indices $\kappa\in \bbN^d$ satisfying $|\kappa|\in\{1,2\}$,
and all $j=1, \dots, k_0$. 
Therefore, their absolute sum also vanishes, implying
$$
\frac 1 {D_n} \sum_{j: |\calA_j|=1}
\sum_{|\kappa|=1}^2|b_{\kappa,j}|
 = {\frac 1 {D_n}}
\sum_{j:|\calA_j|=1}\sum_{i\in \calA_j}
p_i^n\left(\norm{\mu_i^n-\mu_j^0}_1 + 
\frac 1 2 \norm{\Sigma_i^n-\Sigma_j^0}_1\right) \to 0$$
The assumption of Case 3.3, together
with the topological equivalence of the norms
$\norm\cdot_1$ and $\norm\cdot_2$, then implies
$$\frac 1 {D_{n,3}}
\sum_{j:|\calA_j|=1}\sum_{i\in \calA_j}
p_i^n\left(\norm{\mu_i^n-\mu_j^0} + 
\norm{\Sigma_i^n-\Sigma_j^0}\right)
 \to 0,$$
 which is a clear contradiction.
 
\textbf{Case 3.4: $D_{n,4}/D_n  \not\to 0$.}
In this case, it is clear that the coefficients
$c_j/D_{n,4} \not\to 0$, whence
$c_j/D_n \not \to 0$, for all $j=1, \dots, k_0$,
and we immediately obtain a contradiction.

We have thus shown that each
of Cases 3.1-3.4 lead to a contradiction.
We conclude that
at least one of the coefficients $a_{\tau,j}/D_n, b_{\kappa,j}/D_n,
c_j/D_n$ does not tend to zero. 

\textbf{Step 4: Reduction to Location-Gaussian
Strong Identifiability.}
Let $m_n$ denote the maximum of the absolute
values of the coefficients $a_{\tau,j}/D_n, b_{\kappa,j}/D_n,
c_j/D_n$, and set $d_n=1/m_n$. Similarly as in the proof
of Lemma~\ref{lemma:key_bound_strong_identifiability},
there exist 
real numbers $\zeta_{\tau,j},\xi_{\kappa,j},\nu_j$
not all zero 
such that for almost all $x \in \bbR$, 
\begin{align*}
\frac{d_n A_n(x)}{D_n} &\longrightarrow \sum_{j:|\calA_j|>1}
 \sum_{|\tau|=1}^{2\rbar(|\calA_j|)} \zeta_{\tau,j}
\frac{\partial^{|\tau|} f}
{\partial\mu^{\tau}}(x|\theta_j^0), \\
\frac{d_n B_n(x)}{D_n} &\longrightarrow \sum_{j:|\calA_j|=1}
\sum_{|\kappa|=1}^{2} 
 \xi_{\kappa,j}
\frac{\partial^{|\kappa|} f}
{\partial\mu^{\kappa}}(x|\theta_j^0), \\
\frac{d_n C_n(x)}{D_n} &\longrightarrow \sum_{j=1}^{k_0} 
\nu_j f(x|\theta_j^0).
\end{align*}
Furthermore, by Step 3, $\sup_{n\geq 1}d_n < \infty$,
and by the assumption $V(p_{G_n}, p_{G_0}) / D_n \to 0$,
we arrive at 
$$d_n \frac{V(p_{G_n}, p_{G_0})}{D_n} \asymp 
\int d_n \left|\frac{A_n(x) + B_n(x) + C_n(x)}{D_n}\right| dx \to 0.$$
By Fatou's Lemma, the integrand of the above
display vanishes for almost all $x \in \bbR$,
whence
\begin{align*} 
\sum_{j:|\calA_j|>1}  \sum_{|\tau|=1}^{2\rbar(|\calA_j|)} \zeta_{\tau,j}
\frac{\partial^{|\tau|} f}
{\partial\mu^{\tau}}(x|\theta_j^0) +
\sum_{j:|\calA_j|=1}\sum_{|\kappa|=1}^{2} 
 \xi_{\kappa,j}
\frac{\partial^{|\kappa|} f}
{\partial\mu^{\kappa}}(x|\theta_j^0) +
 \sum_{j=1}^{k_0} 
\nu_j f(x|\theta_j^0) = 0.
\end{align*}
The strong identifiability of the location-Gaussian family
now implies that the coefficients
$\zeta_{\tau,j}, \xi_{\kappa,j}, \nu_j$
are all zero, which is a contradiction.
The claim follows.\hfill $\square$
\subsection{Proof of Theorem~\ref{theorem:uniform_rates_strong_identifiable}}
\label{subsec:proof:theorem:uniform_rates_strong_identifiable}
For any mixing measure $G \in \calO_k(\Theta)$, 
let
$F(x,G) = \int_{-\infty}^x p_G(x)d\nu(x)$
denote the CDF of $p_G$. 
Similarly to the previous subsections, 
the proof will follow
from the following key inequality
relating $\widetilde W$ to a statistical
distance, which we take to be 
the Kolmogorov-Smirnov distance
by analogy with~\citet{heinrich2018}. 
\begin{lemma}
\label{lemma:uniform_bound}
Under the same conditions as Theorem~\ref{theorem:uniform_rates_strong_identifiable}, there exist  $C,\epsilon_0 > 0$ depending on $c_{0}, \calF$, and $G_{*}$ such that 
\begin{align}
   \|F(\cdot,G) - F(\cdot,G')\|_\infty \geq C \widetilde{W}(G,G'), \label{key_uniform_inequality}
\end{align}
for any $G \in \mathcal{O}_{k,c_0}(\Theta)$ and $G' \in \mathcal{E}_{k_{0},c_0}(\Theta)$ such that $\widetilde{W}(G, G_{*}) \vee \widetilde{W}(G',G_{*}) \leq \epsilon_{0}$. 
\end{lemma}
Taking Lemma~\ref{lemma:uniform_bound} for granted, notice that
$$\|F(\cdot, \hat G_n) - F(\cdot, G_0^n)\|_\infty
\leq h(p_{\hat G_n}, p_{G_0^n}).$$
Furthermore, under the conditions of Theorem~\ref{theorem:uniform_rates_strong_identifiable},
we may apply
Proposition~\ref{prop:bded_mixing} to deduce that 
there is an event $A_n$ and a constant $c > 1$ such that 
$\bbP(A_n^\cp) \leq c/n$ and 
$\hat p_i^n \geq 1/c$ for all $i\in[\hat k_n]$ over $A_n$. 
As in the proof of Theorem~\ref{theorem:rate_MLE_Gaussian}, 
we may therefore set $c_0' = c_0 \wedge c^{-1}$
and deduce that 
$\hat G_n \in \calO_{k,c_0'}(\Theta)$ over the event $A_n$. 
Therefore, by   Lemma~\ref{lemma:uniform_bound}
and Theorem~\ref{thm:density_estimation_rate_penalized}(ii),
\begin{align*}
\bbE \widetilde W(\hat G_n, G_0^n)
 = \bbE\left[ \widetilde W(\hat G_n, G_0^n) I_{A_n}\right] + 
    \bbE\left[\widetilde W(\hat G_n, G_0^n) I_{A_n^\cp}\right]  
 \lesssim  \bbE\left[ h(p_{\hat G_n}, p_{G_0^n}) I_{A_n}\right] + 1/n \lesssim \log n / \sqrt n.
\end{align*}
This proves the claim; it thus remains to prove the key Lemma~\ref{lemma:uniform_bound}. 

{\bf Proof of Lemma~\ref{lemma:uniform_bound}.} The proof of Lemma~\ref{lemma:uniform_bound} is a refinement of the proof of Theorem 6.3 in~\citet{heinrich2018} where we  carefully consider the behavior of individual mixing components and weights of the mixing measures involved. Notice that in the special case $k^* = 1$, the loss function
$\widetilde W$ is equal to $W_{k + k_0-1}^{k + k_0-1}$, and the claim
can be deduced identically as in~\citet{heinrich2018}. 
We therefore assume $k^* \geq 2$ throughout the sequel. 

To prove  inequality~\eqref{key_uniform_inequality}, we assume that it
does not hold. Therefore,  there exist sequences  $G_{n} \in \mathcal{O}_{k,c_0}(\Theta)$, $G_{n}' \in \mathcal{E}_{k_{0},c_0}(\Theta)$ such that $\widetilde{W}(G_{n}, G_{*}) \to 0$, $\widetilde{W}(G_{n}', G_{*}) \to 0$, and $\|F(\cdot,G_{n}) - F(\cdot, G_{n}')\|_\infty/\widetilde{W}(G_{n},G_{n}') \to 0$ as $n \to \infty$.
Similarly to the proof of
Theorem~\ref{lemma:key_bound_strong_identifiability}, we can find subsequences of $G_{n}, G_{n}'$ such that 
$\mathcal{A}_{j}(G_{n}),\mathcal{A}_{j}(G_{n}')$ do not change with $n\geq 1$, for all $1 \leq j \leq k_{0}$. Without loss of generality, we therefore
assume that $\mathcal{A}_{j} = \mathcal{A}_{j}(G_{n})$ and $\mathcal{A}_{j}' = \mathcal{A}_{j}(G_{n}')$ 
are constant with $n\geq 1$, for all $j \in [k_{0}]$. 
Furthermore, 
up to taking subsequences once again, 
we may assume that $G_n$ has exact order $\bar k \leq k$ for
all $n \geq 1$, and we denote $G_{n} = \sum_{i = 1}^{\bar{k}} p_{i}^{n} \delta_{\theta_{i}^{n}}$ and $G_{n}' = \sum_{i = 1}^{k_{0}} (p_{i}^{n})' \delta_{(\theta_{i}^{n})'}$. 
Now, define
$$(\omega_{i}^{n}, \nu_{i}^{n}) 
 = \begin{cases} 
 (p_{i}^{n}, \theta_{i}^{n}), & 1 \leq  i \leq \bar{k} \\ (-(p_{i - \bar{k}}^{n})', (\theta_{i - \bar{k}}^{n})'),
 & \bar{k} + 1 \leq  i \leq \bar{k} + k_{0},
 \end{cases}$$
and let $\mathcal{B}_{j} = \mathcal{A}_{j}' + \bar{k} = \{i + \bar{k}: i \in \mathcal{A}_{j}'\}$. Based on this notation, we may rewrite $\widetilde{W}(G_{n}, G_{n}')$ as follows:
\begin{align*}
    \widetilde{W}(G_{n}, G_{n}') = \inf_{\boldsymbol{q} \in \Pi(G_{n}, G_{n}')} \biggr\{ \sum_{l = 1}^{k_{*}} \sum_{(i, j) \in \mathcal{A}_{l} \times \mathcal{B}_{l}} q_{i(j - \bar{k})} |\nu_{i}^{n} - \nu_{j}^{n}|^{|\mathcal{A}_{l}| + |B_{l}|  - 1} + \sum_{(i, j) \not \in \cup_{l = 1}^{k_{*}} \mathcal{A}_{l} \times \mathcal{B}_{l}} q_{i(j - \bar{k})} \biggr\}.
\end{align*}

From Lemma 7.1 in~\citet{heinrich2018}, we can find a finite number $(S+1)$ of scaling sequences $0 \equiv \tau_{0}(n) < \tau_{1}(n) < \ldots < \tau_{S}(n) \equiv 1$, where $\tau_{s}(n) = o(\tau_{s + 1}(n))$, such that 
for any $j, j' \in \{1, 2, \ldots, \bar k + k_0 \}$, we can find a unique
integer $s(j, j') \in \{0, 1, \ldots, S\}$ satisfying $|\nu_{j}^n - \nu_{j'}^n| \asymp \tau_{s(j, j')}(n)$. 
In the sequel, we shall sometimes omit the dependence on $n$
in the preceding notation. 
It can be inferred from its definition
that $s(\cdot,\cdot)$ defines an ultrametric on the set $\{1, 2, \ldots, \bar{k} + k_{0} \}$.
As in~\citet{heinrich2018}, this allows us to construct
a coarse-graining tree over the set of balls
in $\{1,\dots,\bar k + k_0\}$ relative to the metric $s$. 
In the interest of being self-contained, we recall their definition
as follows.
\begin{definition}[Definition 7.2 \citep{heinrich2018}]
The coarse-graining tree $\calT$
is the collection of distinct balls
$J = \{i \in \{1, \dots, \bar k + k_0\}: s(i,j) \leq s\}$, called
nodes, for $j=1, \dots, \bar k + k_0$ and $s = 0, \dots, S$. 
Moreover, 
\begin{itemize}
\item The root of $\calT$ is $\calJ_{\mathrm{root}} = \{1, \dots, \bar k + k_0\}$. 
\item $J^\uparrow \in \calT$ is called the parent of a node
$J \in \calT$ if the following implication holds
for all $I \in \calT$,
$$(J \subseteq I \subsetneq J^\uparrow, I \in \calT) \Longrightarrow I = J.$$
\item The set of children of a node $J \in \calT$ is 
$\mathrm{Child}(J) = \{I \in \calT: I^\uparrow = J\}$. 
\item The set of descendants of a node $J \in \calT$ is $\mathrm{Desc}(J) = \{I \in \calT: I^\uparrow \subseteq J\}$. 
\item The diameter of a node $J \in \calT$ is $s(J) = \max_{j,j' \in J} s(j,j')$.
\end{itemize}
\end{definition} 
Since $k_* \geq 2$, it is a straightforward consequence of these definitions
that  the cardinality of $\mathrm{Child}(\calJ_{\mathrm{root}})$
is exactly $k_*$, and we shall write
$\mathrm{Child}(\calJ_{\mathrm{root}}) = \{\calJ_1, \dots, \calJ_{k_*}\}.$
Furthermore, note that
\begin{equation} 
\label{eq:child_voronoi}
\calJ_l = \calA_l \cup \calB_l, \quad l=1, \dots, k_*.
\end{equation}
Now, let $\pi_{J} = \sum_{j \in J} \omega_{j}^n$ and $\tau_{J} = \tau_{s(J)}(n)$,
for all $J \in \calT$. We claim that the following key 
asymptotic equivalence holds.
\begin{lemma} 
\label{lem:equivalence_w_tilde}
We have, 
\begin{align}
    \widetilde{W}(G_{n},G_{n}') \asymp \max \left\{\max \limits_{1 \leq l \leq k_{*}} \max_{J \in \mathrm{Desc}(\mathcal{J}_{l})} |\pi_{J}| \tau_{J^\uparrow}^{|\mathcal{A}_{l}| + |\mathcal{B}_{l}| - 1}, \max \limits_{1 \leq l \leq k_{*}} |\pi_{\mathcal{J}_{l}}| \right\} . \label{eq:asymptotic_equivalence}
\end{align}
\end{lemma}  
The proof of Lemma~\ref{lem:equivalence_w_tilde}
is deferred to Section~\ref{app:pf_equivalence_w_tilde}. 
We next show how this
Lemma may be used to lower bound the
expansion of $F(\cdot,G_n)$ around
$F(\cdot,G_n')$. We begin with
the following result, which is a simplified statement
of Lemma 7.4 of~\citet{heinrich2018}.
In the sequel, for any node $J \in \calT$, let 
$\nu_J$ denote an arbitrary but fixed element of 
$\{\nu_j^n:j\in J\}$. 
\begin{lemma}[Lemma 7.4~\citet{heinrich2018}]
\label{lem:uniform_expansion}
For every $l=1, \dots, k_*$, there exists
a vector $a_l = (a_l(p))_{0 \leq p \leq k+k_0}$
and a remainder $R_l$ such that 
for all $x\in \bbR$, 
$$\sum_{j \in \calJ_l} \omega_j F(x,\nu_j^n)
 = \sum_{p=0}^{k+k_0} a_l(p) \tau_{\calJ_l}^p F^{(p)}(x,\nu_{\calJ_l})
  + R_l(x),$$
Furthermore, the following assertions hold.
\begin{enumerate} 
\item[(i)] We have $a_l(0) = \pi_{\calJ_l}$, and, 
$$\|a_l\| \asymp 
\max_{0 \leq p \leq |\calJ_l|-1} |a_l(p)| 
\gtrsim \max_{J \in \mathrm{Desc}(\calJ_l)} 
|\pi_J| \left(\frac{\tau_{J^\uparrow}}{\tau_{\calJ_l}}\right)^{|\calJ_l|-1}.$$
\item[(ii)] We have,
$
\|R_l\|_\infty = o(\|a_l\| \tau_{\calJ_l}^{k+k_0}).$ \end{enumerate} 
\end{lemma}
By Lemma~\ref{lem:uniform_expansion}, we have 
for all $x \in \bbR$, 
\begin{align*}
F(x,G_n) - F(x,G_n')
 = \sum_{l=1}^{k_*} \sum_{j\in \calJ_l} \omega_j F(x,\nu_j^n)
 = \sum_{l=1}^{k_*}  
 \sum_{p=0}^{k+k_0} a_l(p) \tau_{\calJ_l}^p F^{(p)}(x,\nu_{\calJ_l})
  + \sum_{l=1}^{k_*} R_l(x).
\end{align*}
Let $M_{n,l} = \max_{0 \leq p \leq |\calJ_l| - 1}
 |a_l(p)| \tau_{\calJ_l}^p$ 
 for any $l=1,\dots, k_*$, and let $M_n = \max_{1 \leq l \leq k_*} M_{n,l}$. 
 By Lemma~\ref{lem:uniform_expansion}(i), we have 
\begin{equation}
\label{eq:M_nl_1_pf_si}
M_{n,l} \geq |a_l(0)| = |\pi_{\calJ_l}|,
\end{equation}
and additionally,
\begin{equation} 
\label{eq:M_nl_2_pf_si}
M_{n,l} \gtrsim \max_{J \in \mathrm{Desc}(\calJ_l)}
|\pi_J| \left(\frac{\tau_{J^\uparrow}}{\tau_{\calJ_l}}\right)^{|\calJ_l|-1}
\min_{0 \leq p \leq |\calJ_l|-1} \tau_{\calJ_l}^p
 = \max_{J \in \mathrm{Desc}(\calJ_l)}
|\pi_J| \tau_{J^\uparrow}^{|\calJ_l|-1}.
\end{equation}
Let $D_n = \widetilde W(G_n,G_n')$. 
By Lemma~\ref{lem:equivalence_w_tilde}
and equations~\eqref{eq:M_nl_1_pf_si}--\eqref{eq:M_nl_2_pf_si}, we deduce that $M_n/D_n \gtrsim 1$. 
Additionally, by Lemma~\ref{lem:uniform_expansion}(ii), 
we have 
$\|\sum_l R_l \|_\infty = o(M_n)$. 
Therefore, setting $d_n = D_n/M_n$, we obtain that
there exist
finite real numbers $\alpha_{lp} \in \bbR$, 
not all of which are zero, such that,
$$\left\|d_n \frac{F(\cdot,G_n) - F(\cdot,G_n')}{D_n}
 - \sum_{l=1}^{k_*} \sum_{p=0}^{k+k_0} 
 \alpha_{lp} F^{(p)}(\cdot, \theta_l^*)
 \right\|_\infty \to 0.$$
On the other hand, since $d_n \lesssim 1$, we have
by assumption that
$d_n \|F(\cdot,G_n) - F(\cdot,G_n')\|_\infty/D_n \to 0$, thus we must obtain
$$\left\| \sum_{l=1}^{k_*} \sum_{p=0}^{k+k_0}
\alpha_{lp} F^{(p)} (\cdot,\theta_l^*)
\right\|_\infty = 0.$$
By the strong identifiability condition of order
$k+k_0$, it must follow that $\alpha_{lp} = 0$
for all $l=1, \dots, k_*$ and $p=0, \dots, k+k_0$,
which is a contradiction.
The claim thus follows.\qed 

\subsubsection{Proof of Lemma~\ref{lem:equivalence_w_tilde}.} 
\label{app:pf_equivalence_w_tilde}
We first prove the lower bound of equation~\eqref{eq:asymptotic_equivalence}. For any coupling $\boldsymbol{q} \in 
\Pi(G_{n},G_{n}')$ and for any  $J, J'\in \calT$, we denote
\begin{align*}
    W(J, J'; \boldsymbol{q}) = \sum_{l = 1}^{k_{*}} \sum_{(i, j) \in (\mathcal{A}_{l} \cap J) \times (\mathcal{B}_{l} \cap J')} q_{i(j - \bar{k})} |\nu_{i}^{n} - \nu_{j}^{n}|^{|\mathcal{A}_{l}| + |\mathcal{B}_{l}|  - 1} + \sum_{(i, j) \in \mathcal{M}(J, J') \backslash \cup_{l = 1}^{k_{*}} (\mathcal{A}_{l} \cap J) \times (\mathcal{B}_{l} \cap J')} q_{i(j - \bar{k})}, 
\end{align*}
where $\mathcal{M}(J, J') = (J \cap \{1, \ldots, \bar{k}\}) \times (J' \cap \{\bar{k} + 1, \ldots, \bar{k} + k_{*}\})$. From the above definition, we obtain that $\widetilde{W}(G_{n}, G_{n}') = \inf_{\boldsymbol{q} \in \Pi(G_{n}, G_{n}')} W(\mathcal{J}_{\text{root}}, \mathcal{J}_{\text{root}}; \boldsymbol{q})$. Now, for any coupling $\boldsymbol{q}$ between $G_{n}$ and $G_{n}'$ and for any node $J$ in the tree $\mathcal{T}$, we obtain that
\begin{align*}
    W(\mathcal{J}_{\text{root}}, \mathcal{J}_{\text{root}}; \boldsymbol{q}) \geq W(J, J^{c}; \boldsymbol{q}) + W(J^{c}, J; \boldsymbol{q}).
\end{align*}
Since $|v_{i}^{n} - v_{j}^{n}| \gtrsim \tau_{J\uparrow}$ for any $(i, j) \in J \times J^{c}$ or $(i, j) \in J^{c} \times J$, it follows that  
\begin{align}
    W(J, J^{c}; \boldsymbol{q}) + W(J^{c}, J; \boldsymbol{q}) & \gtrsim \sum_{l = 1}^{k_{*}} \left[ \sum_{(i, j) \in (\mathcal{A}_{l} \cap J) \times (\mathcal{B}_{l} \cap J^{c})} q_{i(j - \bar{k})} + \sum_{(i, j) \in (\mathcal{A}_{l} \cap J^{c}) \times (\mathcal{B}_{l} \cap J)} q_{i(j - \bar{k})} \right] \tau_{J^\uparrow}^{|\mathcal{A}_{l}| + |\mathcal{B}_{l}|  - 1} \nonumber \\
    & \hspace{- 6 em} + \left[\sum_{(i, j) \in \mathcal{M}(J, J^{c}) \backslash \cup_{l = 1}^{k_{*}} (\mathcal{A}_{l} \cap J) \times (\mathcal{B}_{l} \cap J^{c})} q_{i(j - \bar{k})} + \sum_{(i, j) \in \mathcal{M}(J^{c}, J) \backslash \cup_{l = 1}^{k_{*}} (\mathcal{A}_{l} \cap J^{c}) \times (\mathcal{B}_{l} \cap J)} q_{i(j - \bar{k})} \right] : = \mathcal{C}, \label{eq:lower_bound_proof}
\end{align}
There are two settings of node $J$: 

\noindent
\textbf{Case 1:} $J \in \mathrm{Child}(\calJ_{\mathrm{root}})$. 
In this case,  $J = \mathcal{J}_{l}$ for some $l \in [k_{*}]$. 
We deduce from equation~\eqref{eq:child_voronoi}
that $\mathcal{A}_{l} \cap J^{c} = \mathcal{B}_{l} \cap J^{c} = \emptyset$. Therefore, from equation~\eqref{eq:lower_bound_proof}, we obtain that
\begin{align}
\mathcal{C} 
 = \sum_{(i, j) \in \mathcal{M}(J, J^{c})} q_{i(j - \bar{k})} + \sum_{(i, j) \in \mathcal{M}(J^{c}, J)} q_{i(j - \bar{k})} 
 \geq \left| \sum_{(i, j) \in \mathcal{M}(J, J \cup J^c)} q_{i(j - \bar{k})} - 
 \sum_{(i, j) \in \mathcal{M}(J\cup J^c, J)} q_{i(j - \bar{k})} 
\right| = |\pi_J|.  
\label{eq:lower_bound_proof_first}
\end{align}
\textbf{Case 2:} $J \in \mathrm{Desc}(\calJ_l)$ for some $l \in [k_{*}]$. Under this case, we can verify that
\begin{align}
    \mathcal{C}
    & \gtrsim \Bigg[ \sum_{(i, j) \in (\mathcal{A}_{l} \cap J) \times (\mathcal{B}_{l} \cap J^{c})} q_{i(j - \bar{k})} + \sum_{(i, j) \in (\mathcal{A}_{l} \cap J^{c}) \times (\mathcal{B}_{l} \cap J)} q_{i(j - \bar{k})} + \sum_{(i, j) \in \mathcal{M}(J, J^{c}) \backslash \cup_{l = 1}^{k_{*}} (\mathcal{A}_{l} \cap J) \times (\mathcal{B}_{l} \cap J^{c})} q_{i(j - \bar{k})} \nonumber \\
    & + \sum_{(i, j) \in \mathcal{M}(J^{c}, J) \backslash \cup_{l = 1}^{k_{*}} (\mathcal{A}_{l} \cap J^{c}) \times (\mathcal{B}_{l} \cap J)} q_{i(j - \bar{k})} \Bigg] \tau_{J^\uparrow}^{|\mathcal{A}_{l}| + |\mathcal{B}_{l}|  - 1} 
     \gtrsim |\pi_{J}| \tau_{J^\uparrow}^{|\mathcal{A}_{l}| + |\mathcal{B}_{l}|  - 1}. \label{eq:lower_bound_proof_third}
\end{align}
Combining the results of equations~\eqref{eq:lower_bound_proof},~\eqref{eq:lower_bound_proof_first}, and~\eqref{eq:lower_bound_proof_third}, we obtain the lower bound that
\begin{align*}
    \widetilde{W}(G_{n},G_{n}') \gtrsim \max \left\{\max \limits_{1 \leq l \leq k_{*}} \max_{J \in \text{Desc}(\mathcal{J}^{l})} |\pi_{J}| \tau_{J^\uparrow}^{|\mathcal{A}_{l}| + |\mathcal{B}_{l}| - 1}, \max \limits_{1 \leq l \leq k_{*}} |\pi_{\mathcal{J}^{l}}| \right\}.
\end{align*}
Therefore, to obtain the conclusion of claim~\eqref{eq:asymptotic_equivalence}, it remains
to verify the upper bound of $\widetilde{W}(G_{n}, G_{n}')$ in that claim. 
Based on Lemma B.2 of~\citet{heinrich2018}, we can construct a coupling $\bar{\boldsymbol{q}}$ between $G_{n}$ and $G_{n}'$ such that for any node $J\in\mathcal{T}$, we have
\begin{align}
\label{eq:qbar_def_min_pJ}
    \sum_{l = 1}^{k_{*}} \sum_{(i, j) \in (\mathcal{A}_{l} \cap J) \times (\mathcal{B}_{l} \cap J)} \bar{q}_{i(j - \bar{k})} = \min \{p_{J}, p_{J}'\},
\end{align}
where $p_{J} = \sum_{i \in J \cap \{1, \ldots, \bar{k}\}} p_{i}^{n}$ and $p_{J}' = \sum_{i \in J \cap \{\bar{k} + 1, \ldots, \bar{k} + k_{0}\}} (p_{i-\bar{k}}^{n})'$. Given the coupling $\bar{\boldsymbol{q}}$, we first prove that for any node $J$ that is a descendant of $\mathcal{J}^{l}$ or equal to $\mathcal{J}^{l}$ for some $l \in [k_{*}]$, we have
\begin{align}
    W(J, J; \bar{\boldsymbol{q}}) \lesssim \max_{K \in \text{Desc}(J)} |\pi_{K}| \tau_{K^\uparrow}^{|\mathcal{A}_{l}| + |\mathcal{B}_{l}| - 1}. \label{eq:upper_bound_proof}
\end{align}
We prove the inequality~\eqref{eq:upper_bound_proof} by induction. When $J$ is an end node of $\mathcal{J}^{l}$, $W(J, J; \bar{\boldsymbol{q}}) = 0$; therefore, inequality~\eqref{eq:upper_bound_proof} holds true. We assume that this inequality holds for any  node 
$K$ which is a child 
of a given node $J$. We now proceed to show that this inequality also holds for $J$. In fact, we have the following identity:
\begin{align*}
    W(J, J; \bar{\boldsymbol{q}}) = \sum_{K \in \text{Child}(J)} \biggr(W(K,K; \bar{\boldsymbol{q}}) + \sum_{K' \neq K; K' \in \text{Child}(J)} W(K, K'; \bar{\boldsymbol{q}}) \biggr).
\end{align*}
Note that, for any $K$ and $K'$ that are children of node $J$, we have
\begin{align*}
    W(K, K'; \bar{\boldsymbol{q}}) = \sum_{(i, j) \in (\mathcal{A}_{l} \cap K) \times (\mathcal{B}_{l} \cap K')} \bar{q}_{i(j - \bar{k})} |\nu_{i}^{n} - \nu_{j}^{n}|^{|\mathcal{A}_{l}| + |\mathcal{B}_{l}|  - 1}.
\end{align*}
From the induction hypothesis, we obtain that $W(K, K; \bar{\boldsymbol{q}}) \lesssim \max_{Q \in \text{Desc}(K)} |\pi_{Q}| \tau_{Q^\uparrow}^{|\mathcal{A}_{l}| + |\mathcal{B}_{l}| - 1}$. Furthermore, for any $K' \neq K$ and $K' \in \text{Child}(J)$, we find that
\begin{align*}
    W(K, K'; \bar{\boldsymbol{q}}) \lesssim \biggr(\sum_{(i, j) \in (\mathcal{A}_{l} \cap K) \times (\mathcal{B}_{l} \cap K')} \bar{q}_{i(j - \bar{k})}\biggr) \tau_{J}^{|\mathcal{A}_{l}| + |\mathcal{B}_{l}|  - 1} \lesssim |\pi_{K}| \tau_{J}^{|\mathcal{A}_{l}| + |\mathcal{B}_{l}|  - 1},
\end{align*}
where the bound on the first factor follows from
equation~\eqref{eq:qbar_def_min_pJ}. 
Collecting the above results, we arrive at $W(J, J; \bar{\boldsymbol{q}}) \lesssim \max_{K \in \text{Desc}(J)} |\pi_{K}| \tau_{K^\uparrow}^{|\mathcal{A}_{l}| + |\mathcal{B}_{l}| - 1}$. Therefore,  inequality~\eqref{eq:upper_bound_proof} is proved for any node $J$ that is a descendant of $\mathcal{J}^{l}$ or equal to $\mathcal{J}^{l}$ for some $l \in [k_{*}]$. 

Now, we proceed to prove the following inequality
\begin{align}
    W(\mathcal{J}_{\text{root}}, \mathcal{J}_{\text{root}}; \bar{\boldsymbol{q}}) \lesssim \max \left\{\max \limits_{1 \leq l \leq k_{*}} \max_{J \in \text{Desc}(\mathcal{J}^{l})} |\pi_{J}| \tau_{J^\uparrow}^{|\mathcal{A}_{l}| + |\mathcal{B}_{l}| - 1}, \max \limits_{1 \leq l \leq k_{*}} |\pi_{\mathcal{J}^{l}}|\right\}. \label{eq:upper_bound_first}
\end{align}
In fact, we have
\begin{align*}
    W(\mathcal{J}_{\text{root}}, \mathcal{J}_{\text{root}}; \bar{\boldsymbol{q}}) = \sum_{l = 1}^{k_{*}} \biggr(W(\mathcal{J}^{l},\mathcal{J}^{l}; \bar{\boldsymbol{q}}) + \sum_{l' \neq l} W(\mathcal{J}^{l}, \mathcal{J}^{l'}; \bar{\boldsymbol{q}}) \biggr).
\end{align*}
From inequality~\eqref{eq:upper_bound_proof}, we obtain that $W(\mathcal{J}^{l},\mathcal{J}^{l}; \bar{\boldsymbol{q}}) \lesssim \max_{J \in \text{Desc}(\mathcal{J}^{l})} |\pi_{J}| \tau_{J^\uparrow}^{|\mathcal{A}_{l}| + |\mathcal{B}_{l}| - 1}$ for any $l \in [k_*]$.
Furthermore, for any $l' \neq l$, we find that
\begin{align*}
    W(\mathcal{J}^{l}, \mathcal{J}^{l'}; \bar{\boldsymbol{q}}) = \sum_{(i, j) \in \mathcal{M}(\mathcal{J}^{l}, \mathcal{J}^{l'}) \backslash \cup_{l = 1}^{k_{*}} (\mathcal{A}_{l} \cap \mathcal{J}^{l}) \times (\mathcal{B}_{l} \cap \mathcal{J}^{l'})} \bar{q}_{i(j - \bar{k})} \lesssim |\pi_{\mathcal{J}^{l}}| = |\pi_{\mathcal{J}^{l}}|.
\end{align*}
Putting the above results together, we obtain the conclusion of inequality~\eqref{eq:upper_bound_first}. Since $\widetilde{W}(G_{n},G_{n}') \leq W(\mathcal{J}_{\text{root}}, \mathcal{J}_{\text{root}}; \bar{\boldsymbol{q}})$, we reach the conclusion of claim~\eqref{eq:asymptotic_equivalence}.\qed

\section{Additional Results}
\label{sec:auxiliary}
In this appendix, we state
and prove the following result
which was deferred from the main text. 
\begin{lemma}
\label{lemma:relation_strong_identifiability}
Let $\Theta \subseteq \bbR^d$ be a compact set
with nonempty interior.
\begin{enumerate} 
\item[(a)] Let $\Delta  = 1\vee \diam(\Theta) < \infty$
and $G_0 \in \calE_{k_0}(\Theta)$. 
Then, for any $G \in \mathcal{O}_{k}(\Theta)$, we have 
$$ \mathcal{D}(G, G_{0}) \geq \frac 1 {\Delta^2} W_2^2(G, G_0).$$

\item[(b)] Assume the mixing measure
$G_0 \in \calE_{k_0}(\Theta)$ admits a support point
$\theta_0$ lying in the interior of $\Theta$.
Then, 
$$\sup_{\substack{G \in \mathcal{O}_{k}(\Theta)\\
G \neq G_0}}
\frac{\mathcal{D}(G, G_{0})}{W_{2}^2(G, G_{0})} =\infty.$$
\end{enumerate} 
\end{lemma}
\begin{proof}
Let $G \in \calO_k(\Theta)$
and $\calA_j = \calA_j(G)$ for all $j=1, \dots, k_0$. 
By Lemma B.2 of~\citet{heinrich2018}, there exists
a coupling $\bar\bq \in \Pi(G, G_0)$ such that
$$\sum_{i\in \calA_j} \bar q_{ij} = p_j^0 \wedge \sum_{i\in \calA_j} p_i,\quad j=1, \dots, k_0.$$
Using the above display and the marginal constraints
in the definition of a coupling, we obtain
\begin{align} 
\nonumber 
W_2^2(G, G_0)
 &\leq \sum_{i=1}^k \sum_{j=1}^{k_0}\bar q_{ij}
\|\theta_i - \theta_j^0\|^2 \\
 \nonumber 
&\leq \sum_{j=1}^{k_0} \sum_{i\in \calA_j} \bar q_{ij}
\|\theta_i - \theta_j^0\|^2
 + \Delta^2 \sum_{j=1}^{k_0} \sum_{i\not\in \calA_j} \bar q_{ij}\\
\nonumber 
&= \sum_{j=1}^{k_0} \sum_{i\in \calA_j} \bar q_{ij}
\|\theta_i - \theta_j^0\|^2
 + \Delta^2\sum_{j=1}^{k_0}\left[ p_j^0 -  \sum_{i\in \calA_j} \bar q_{ij}\right]\\
 \label{eq:step_W2_D_bound}
 &\leq \sum_{j=1}^{k_0} \sum_{i\in \calA_j} p_i
\|\theta_i - \theta_j^0\|^2
 + \Delta^2 \sum_{j=1}^{k_0}\left| p_j^0 -  \sum_{i\in \calA_j} p_i\right| \\
\nonumber 
&\leq \sum_{j: |\calA_j| = 1} \sum_{i\in \calA_j} p_i
\|\theta_i - \theta_j^0\|^2 + 
\sum_{j:|\calA_j|\geq 2} \sum_{i\in \calA_j} p_i
\|\theta_i - \theta_j^0\|^2
 + \Delta^2 \sum_{j=1}^{k_0}\left| p_j^0 -  \sum_{i\in \calA_j} p_i\right| \\
\nonumber 
 &\leq \Delta\sum_{j: |\calA_j| = 1} \sum_{i\in \calA_j} p_i
\|\theta_i - \theta_j^0\| + 
\sum_{j:|\calA_j|\geq 2} \sum_{i\in \calA_j} p_i
\|\theta_i - \theta_j^0\|^2
 + \Delta^2 \sum_{j=1}^{k_0}\left| p_j^0 -  \sum_{i\in \calA_j} p_i\right| \\
\nonumber 
 &\leq \Delta^2\left\{ \sum_{j: |\calA_j| = 1} \sum_{i\in \calA_j} p_i
\|\theta_i - \theta_j^0\| + 
\sum_{j:|\calA_j|\geq 2} \sum_{i\in \calA_j} p_i
\|\theta_i - \theta_j^0\|^2
 +  \sum_{j=1}^{k_0}\left| p_j^0 -  \sum_{i\in \calA_j} p_i\right|\right\} \\
 & = \Delta^2 \calD(G, G_0),
 \end{align}
 since $\Delta \geq 1$. 
This proves part (a). To prove part (b), 
recall that $G_0 = \sum_{j=1}^{k_0} p_j^0 \delta_{\theta_j^0}$ admits a support
point lying in the interior of $\Theta$. Without
loss of generality, we assume this support
point is $\theta^0_1$. Therefore, 
there exists $\epsilon_0 > 0$ such
that for all $\epsilon \in (0,\epsilon_0)$, 
$\theta_\epsilon^0 := (1+\epsilon)\theta_1^0 \in \Theta$. 
Define the mixing measure
$$G_\epsilon =p_1^0 \delta_{\theta_\epsilon^0}
 + \sum_{j=2}^{k_0} p_j^0 \delta_{\theta_j^0}
 \in \calO_{k_0}(\Theta) \subseteq\calO_k(\Theta).$$
Clearly, we may also choose $\epsilon_0$
small enough such that $\theta_\epsilon^0 \in \calA_1(G_\epsilon)$
for all $\epsilon\in (0,\epsilon_0)$. Thus, $|\calA_j(G_\epsilon)|=1$
for every $j=1,\dots,k_0$. 
By equation~\eqref{eq:step_W2_D_bound}, we therefore have 
$$W_2^2(G_\epsilon, G_0) 
\leq p_1^0 \|\theta_\epsilon^0 - \theta_1^0\|^2
 = p_1^0 \epsilon^2.$$
On the other hand, using again
the fact that $|\calA_j(G_\epsilon)|=1$ for each $j=1,\dots, k_0$, we have
$$\calD(G_\epsilon, G_0)
 = p_1^0 \epsilon.$$
 We deduce that
$$\sup_{\substack{G \in \mathcal{O}_{k}(\Theta)\\
G \neq G_0}}
\frac{\mathcal{D}(G, G_{0})}{W_{2}^2(G, G_{0})} 
\geq \sup_{\epsilon \in(0, \epsilon_0)}
\frac{\mathcal{D}(G_\epsilon, G_{0})}{W_{2}^2(G_\epsilon, G_{0})}\geq \sup_{\epsilon\in(0, \epsilon_0)}
\frac 1 \epsilon = \infty,$$
as claimed.
\end{proof}

\section{Simulation Study}
\label{sec:simulations}

We perform a simulation study to illustrate
the convergence rates of the penalized MLE
given in Sections~\ref{sec:pointwise_settings}
and~\ref{sec:uniform_settings}. 
All simulations hereafter
were performed in Python 3.7 on a standard Unix machine,
and we provide further numerical details in Appendix~\ref{app:numerical}. 
All code for reproducing
our simulation study is publicly available.\footnote{\href{https://github.com/tmanole/Refined-Mixture-Rates}{https://github.com/tmanole/Refined-Mixture-Rates}} 

We consider three models A--C, 
which respectively correspond to the settings described
in Sections~\ref{sec:pointwise_strongly_identifiable}, \ref{sec:pointwise_weakly_identifiable},
and~\ref{sec:uniform_settings}. 
In each case, we choose the kernel density $f$ to be the 
$d$-dimensional Gaussian density, 
and we generate observations from the Gaussian mixture density,
$$p_{G_0}(x) = \sum_{j=1}^{k_0} \pi_{j}^0 \frac {\exp\left\{ - \frac 1 2 (x-\mu_j^0)^\top (\Sigma_j^0)^{-1} (x-\mu_j^0)\right\}}{\sqrt{\det(2\pi\Sigma_j^0)}},$$
where $x \in \bbR^d$.  
The models are defined as follows. 

{\bf Model A.} We treat
the scale parameters as equal and known, 
and set 
\begin{equation} 
\label{eq:equal_scales} 
\Sigma_1^0 = \dots \Sigma_{k_0}^0 = .01 I_d,
\end{equation}
with $d=2$ and $k_0 = 2$. 
The resulting location-Gaussian family of densities
is strongly identifiable~\citep{chen1995, ho2016EJS}, 
thus the result of Theorem~\ref{theorem:rate_MLE_strong_identifiability}
applies to this family. 
We set the location parameters and mixing proportions as follows,
\begin{alignat*}{3}
    \theta_1^0 = \begin{pmatrix} 0\\0\end{pmatrix},~~ 
    \theta_2^0 = \begin{pmatrix}.2\\.2\end{pmatrix},~~
    \pi_1^0 = \pi_2^0 = \frac 1 2. 
\end{alignat*}    

{\bf Model B.} We next consider a two-dimensional
Gaussian mixture model with $k_0 = 3$ components,
however we now treat both location  and scale parameters as unknown.
Define,
\begin{alignat*} {1}
&    \mu_1^0 = \begin{pmatrix} 0\\.3\end{pmatrix},~~ 
    \mu_2^0 = \begin{pmatrix}.1\\-.4 \end{pmatrix},~~
    \mu_3^0 = \begin{pmatrix}.5\\.2 \end{pmatrix},\\ 
&   \Sigma_1^0 = \begin{pmatrix}.042824 & .017324 \\ .017324 & .081759\end{pmatrix},~~
     \Sigma_2^0 = \begin{pmatrix}.0175 & -.0125 \\ -.0125 & .0175\end{pmatrix},\\
&     \Sigma_3^0 = \begin{pmatrix}0.01 & -.0125 \\ -.0125 & .0175\end{pmatrix},~~
     \pi_1^0 = \frac 1 3, \pi_2^0 = \frac 1 4, \pi_3^0 = \frac 1 3.
\end{alignat*}     
The above parameters are taken from
the simulation study
of~\citet{ho2016convergence}, 
up to rescaling. This model falls within
the setting of Theorem~\ref{theorem:rate_MLE_Gaussian}. 

{\bf Model C.} 
We again consider a location-Gaussian family as in Model A, 
but now with parameters
$G_0 \equiv G_0^n$ depending on the sample size $n$. 
We set the scale parameters as in equation~\eqref{eq:equal_scales}
with $d=1$. 
Furthermore, we consider two distinct 
submodels, depending on 
the true number of components $k_0$. 
Our definitions depend on 
the sequence $\epsilon_n = n^{- \frac 1 {4k_0-6}}$.
\begin{itemize} 
\item When $k_0=3$, we set
\begin{alignat*}{1} 
\mu_{1,n}^0 = 0, ~~ 
\mu_{2,n}^0 = .2 + \epsilon_n, ~~
\mu_{3,n}^0 = .2 + 4\epsilon_n.
\end{alignat*}
\item When $k_0 = 4$, we retain the above parameters
and additionally define
\begin{alignat*}{1}
\mu_{4,n}^0 = .2 - 1.5\epsilon_n.
\end{alignat*}
\end{itemize}
In both cases, the mixing proportions
are chosen such that the resulting mixtures are balanced.
These models correspond to the setting described
in Section~\ref{sec:uniform_settings}, relative to the 
limiting mixing measure
$$G_* = \frac 1 2 \delta_0 + \frac 1 2\delta_{.2},
\quad k_* = 2.$$
\begin{figure}[t]
\centering
\begin{tabular}{ccc}
\includegraphics[width=0.3\textwidth]{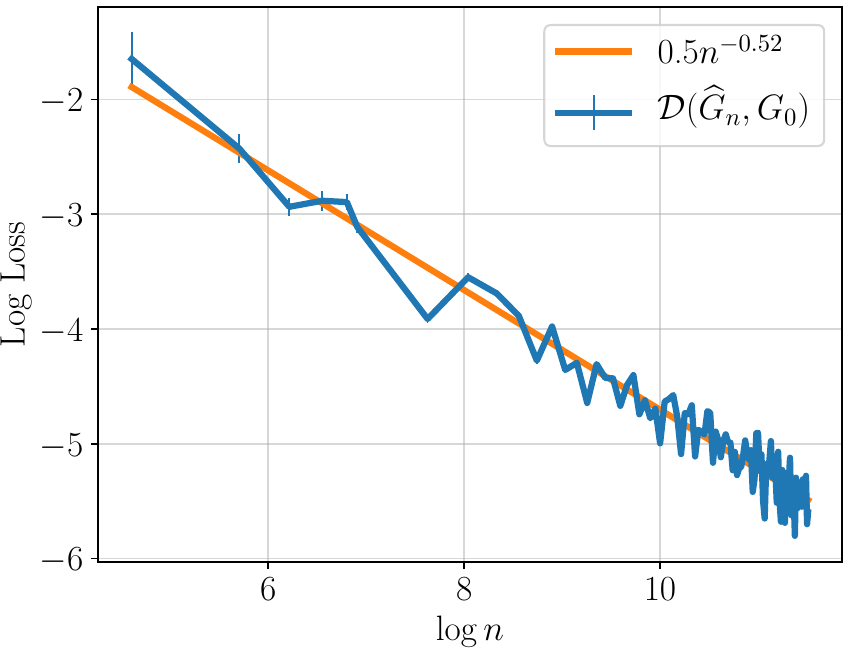}     &
\includegraphics[width=0.3\textwidth]{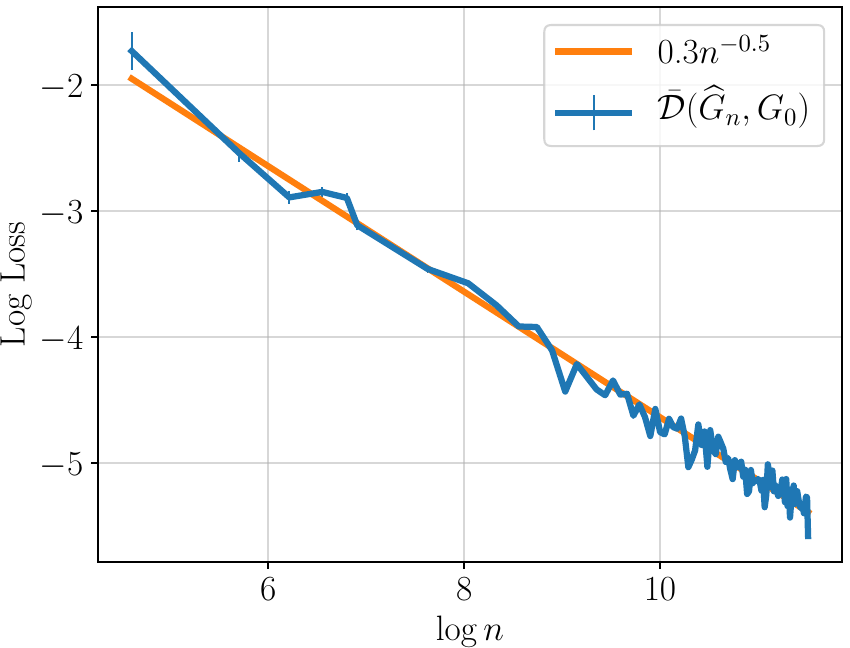} & 
\includegraphics[width=0.3\textwidth]{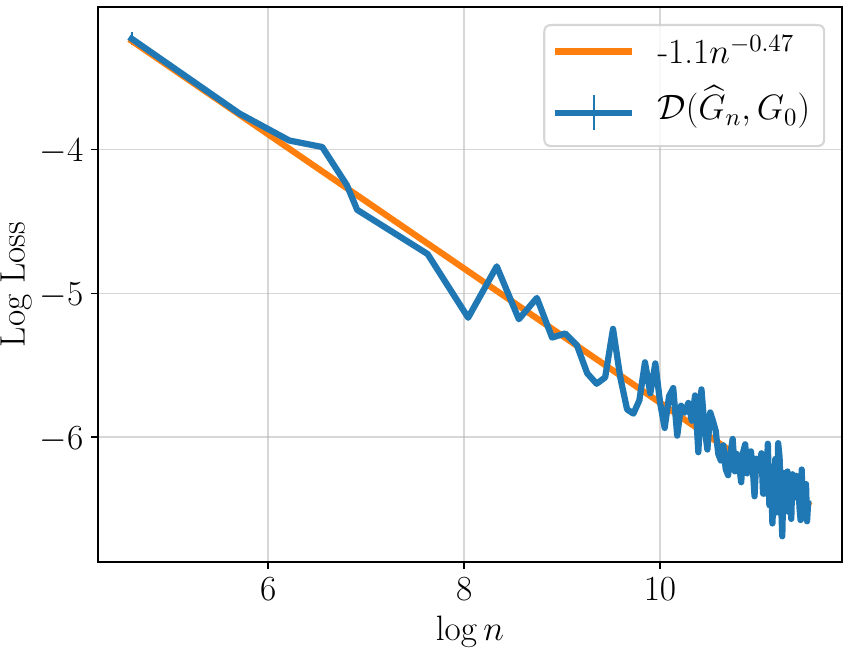} \\
(a) Model A, $k=3$ & (c) Model B, $k=4$ & (e) Model C, $k=k_0=3$\\[0.1in]
\includegraphics[width=0.3\textwidth]{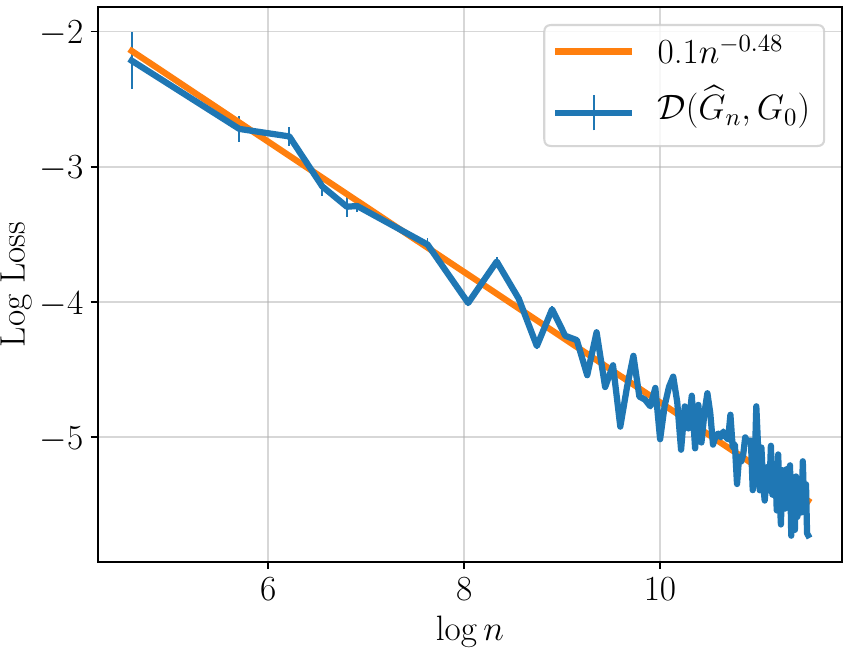}     &
\includegraphics[width=0.3\textwidth]{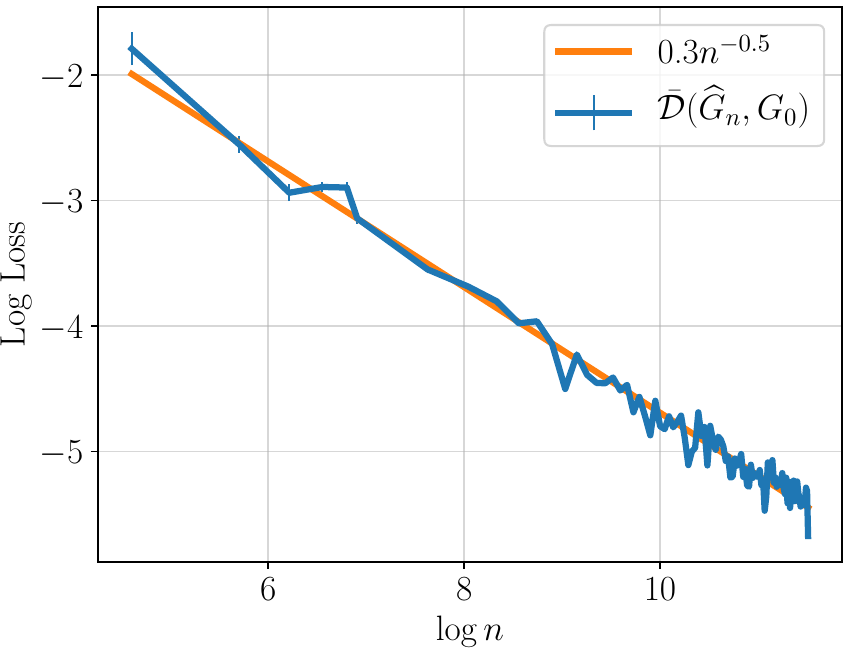} &
\includegraphics[width=0.3\textwidth]{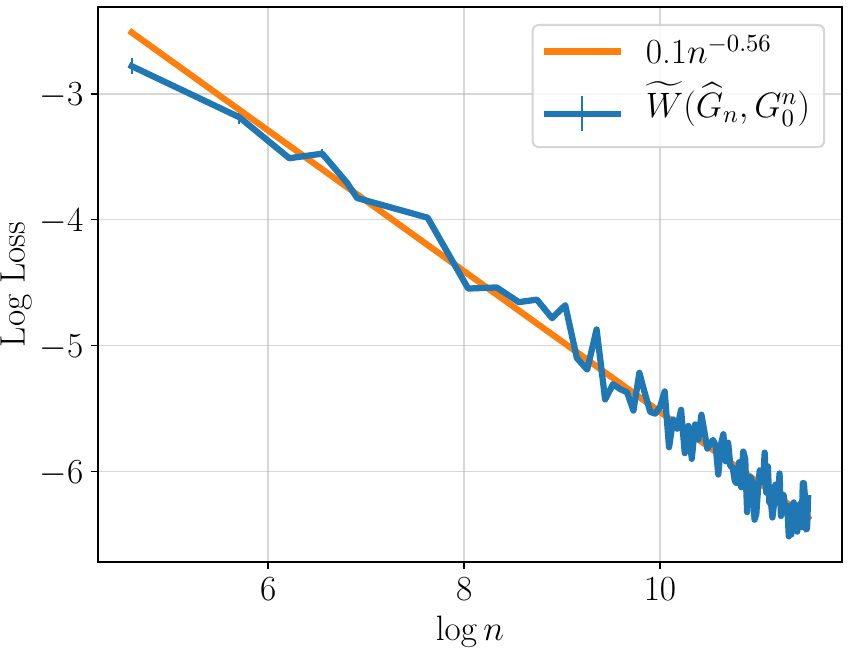}\\ 
(b) Model A, $k=4$  & (d) Model B, $k=5$ & (f) Model C, $k=k_0=4$
    \end{tabular}
    \caption{Log-log scale plots for the simulation results under Models A--C.
    For each model and sample size $n$, we compute the estimator $\hat G_n$ 
    on 20 independent samples of size $n$. Its average discrepancy from the true
    mixing measure
    is plotted in blue, 
    with error bars representing two empirical standard deviations. We additionally plot, in orange, the fitted linear regression line
    of these points, obtained using
    the method of least squares.} 
    \label{fig:sim}
\end{figure} 
For each model, we generate
20 samples of size $n$, for 100 different choices of $n$ 
between $10^2$ and $10^5$. 
For each sample, we compute the penalized MLE $\hat G_n$
with respect to the tuning parameter $\xi_n = \log n$,
and with respect to a number of components
$k$. For the fixed Models A--B, we choose $k \in \{ k_0 + 1, k_0 + 2\}$,
whereas for the varying Model~C, we choose $k =k_0 \in \{k^*+1, k^*+2\}$. 

We report in Figure~\ref{fig:sim}
the average discrepancy between $\hat G_n$
and $G_0$ for each model and choice of $k$. 
The discrepancies are respectively taken to be
$\calD, \widebar\calD$ and $\widetilde W$ for Models A--C. 
In each case, it can be seen that the average discrepancy
from $\hat G_n$ to $G_0$ decays approximately
at the rate $n^{-1/2}$, as was anticipated
by Theorems~\ref{theorem:rate_MLE_strong_identifiability}, 
~\ref{theorem:rate_MLE_Gaussian}
and~\ref{theorem:uniform_rates_strong_identifiable}. 

While these empirical  convergence rates 
are similar across the three models, 
they imply vastly different  convergence behaviors for the 
individual fitted parameters. For example, Figure~\ref{fig:sim}(a)
implies that $\hat G_n$ has exactly two 
location parameters $\hmu_j^n$
which converge 
to one of their population counterparts at the approximate rate $\alpha_n = n^{-1/4}$, and 
a third location parameter
converging at the faster rate $\beta_n = n^{-1/2}$. 
Under Figure~\ref{fig:sim}(e), 
a similar conclusion holds true, but 
now
two possibilities arise: either $\alpha_n = n^{-1/6}$ and $\beta_n = n^{-1/2}$, or $\alpha_n = \beta_n = n^{-1/4}$. 
In contrast, past
literature on mixture models
only implies that the worst of these rates (i.e. $n^{-1/4}$ for Model A 
and $n^{-1/6}$ for Model C) hold for {\it all three} fitted parameters.
The main contribution of our work was to show
that such results are overly pessimistic, and that the 
fitted parameters of finite mixture models typically
enjoy heterogeneous rates of convergence. 
In particular, a subset of the estimated parameters in finite mixture
models may converge as fast as the parametric rate.

\subsection{Numerical Specifications}
\label{app:numerical}

We implement the penalized MLE $\hat G_n$
using Algorithm~\ref{alg:em},
which
is a slight modification of the 
EM algorithm~\citet{dempster1977EM}
accounting for the penalty on the mixing proportions. 
This algorithm was previously discussed, 
for instance, 
by~\citet{chen2008a, manole2021b}, and only differs from
the traditional EM algorithm for Gaussian mixture models
through the update on line~6. 
We used Algorithm~\ref{alg:em} 
as written for Model~B, 
whereas for Models~A and~C, we omitted the update
on line 8 for the scale parameters, and simply
held them fixed to their true values. 
\begin{algorithm}[H]
\SetAlgoNoLine
\init{Starting values $\Psi^{(0)} = (\theta_1^{(0)},
\dots, \theta_k^{(0)}, 
\Sigma_1^{(0)}, \dots, \Sigma_k^{(0)}, \pi_1^{(0)}, \dots, \pi_k^{(0)})$;
i.i.d. sample $X_1, \dots, X_n$;
tuning parameter $\xi_n = \log n$;
maximum number of iterations $T > 0$;
convergence criterion $\epsilon > 0$.}

\Repeat{$\left\|\Psi^{(t)} - \Psi^{(t-1)}\right\| \leq \epsilon ~\mathrm{ or } ~t \geq T$.}{
\kwEstep{}{Compute
$w_{ij}^{(t+1)} \leftarrow
	\frac{\pi_j^{(t)} \log  f(X_i; \theta_j^{(t)} , \Sigma_j^{(t)})}
	     {\sum_{l=1}^k \pi_l^{(t)} \log  f(X_i; \theta_l^{(t)},\Sigma_l^{(t)}) },
  ~~ i=1, \dots, n; \ j = 1, \dots, k.$
}{}
\kwMstep{}{
For $j = 1, \dots, k$, 

  $
\pi_j^{(t+1)} \leftarrow  \frac{ \sum_{i=1}^n w_{ij}^{(t)} + \xi_n}{n + k\xi_n}, $
  
  $\mu_j^{(t+1)} \leftarrow \sum_{i=1}^n w_{ij}^{(t)} X_i / \sum_{i=1}^n w_{ij}^{(t)},$
  
  $\Sigma_j^{(t+1)} \leftarrow 
  \sum_{i=1}^n w_{ij}^{(t)} (X_i - \mu_j^{(t)})(X_i - \mu_j^{(t)})^\top / \sum_{i=1}^n w_{ij}^{(t)},$
   
  $\Psi^{(t+1)} \leftarrow (\theta_1^{(t)},
\dots, \theta_k^{(t)}, \Sigma_1^{(t)}, \dots\Sigma_k^{(t)}, \pi_1^{(t)}, \dots, \pi_k^{(t)}),$

 $t \leftarrow t+1.$
}{}
}\output{$\Psi^{(t)}$.}
\caption{Modified EM Algorithm.
\label{alg:em}}
\end{algorithm}

We chose the convergence criteria
$\epsilon = 10^{-8}$ and $T = 2,000$.  
Since our aim is to illustrate
theoretical properties of the estimator
$\hat G_n$, 
we initialized the EM algorithm favourably.
In particular, for any given $k$ and $k_0$, 
and for each replication, we randomly
partitioned the set $\{1,\dots,k\}$
into $k_0$ index sets $I_1, \dots, I_{k_0}$, each containing at least
one point. We then sampled $\theta_j^{(0)}$
(resp. $\Sigma_j^{(0)}$) 
from a Gaussian distribution 
with vanishing covariance, centered
at $\theta_{\ell}^0$ 
(resp. $\Sigma_{\ell}^0$), where 
$\ell$ is the unique index 
such that $j \in I_\ell$.

\end{appendix}

\end{document}